\documentclass[12pt,a4paper]{article}
\usepackage[T1]{fontenc}
\usepackage{lmodern}
\usepackage[latin1]{inputenc}
\usepackage{latexsym}                 
\usepackage{color}
\usepackage{amssymb}
\usepackage{amsmath}
\usepackage{amsthm}
\usepackage{graphicx}
\usepackage{verbatim}
\usepackage{mathptmx}      
\def\eref#1{(\ref{#1}%
)}
\def\fa{\hbox{ for all }}

\def\dfrac#1#2{\displaystyle{\frac{#1}{#2}   }}

\def\b1{\mathbf 1}

\newcommand{\R}{\ensuremath{\mathbb{R}}}

\newcommand{\N}{\ensuremath{\mathbb{N}}}

\def\call{\mathcal{L}}

\def\P{\mathcal{P}}

\newtheorem{definition}{Definition}

\newtheorem{theorem}{Theorem}
\newtheorem{corollary}{Corollary} 

\def\Pqd{{\cal P}_q^d}
\begin{document}
\begin{center}
{\bf Optimal Stencils in Sobolev Spaces}

Oleg Davydov\footnote{%
Univ. Gie\ss{}en, Oleg.Davydov@math.uni-giessen.de\\
https://www.staff.uni-giessen.de/odavydov/}
and Robert Schaback\footnote{%
Univ. G\"ottingen,
schaback@math.uni-goettingen.de\\
http://num.math.uni-goettingen.de/schaback/research/group.html}\\

Draft of \today
\end{center}
{\bf Abstract}: 
This paper proves that the approximation of pointwise derivatives of order $s$
of functions in Sobolev space $W_2^m(\R^d)$ by linear combinations
of function values cannot have a
convergence rate better than $m-s-d/2$, no matter how many nodes are used
for approximation and where they are placed.
These convergence rates are attained by
{\em scalable} approximations that are exact on
polynomials of order at least $\lfloor m-d/2\rfloor +1$, proving that the
  rates are optimal for given $m,\,s,$ and $d$.
And, for a fixed node set $X\subset\R^d$,
the convergence rate in any Sobolev space $W_2^m(\Omega)$ 
cannot be better than $q-s$ 
where $q$ is the maximal possible order of polynomial exactness 
of approximations based on $X$, no matter how large $m$ is.
In particular,
  scalable stencil constructions via polyharmonic kernels
  are shown to realize the optimal convergence rates, and good approximations
  of their error
  in Sobolev space can be  calculated
  via their error in Beppo-Levi spaces.
This allows to construct near-optimal stencils
in Sobolev spaces stably and efficiently,
for use in meshless methods to solve partial differential equations
via generalized finite differences (RBF-FD).
Numerical examples are included for illustration.
\section{Introduction}\label{SecIntro}
We consider discretizations of continuous 
linear functionals $\lambda\;:\;U\to\R$ on some normed linear space $U$
of real-valued functions on some bounded domain $\Omega\subset\R^d$.
The discretizations are {\em nodal}, i.e. they work with values $u(x_j)$ 
of functions $u\in U$ on a set $X=\{x_1,\ldots,x_M\}
\subset\Omega$ of {\em nodes} by 
\begin{equation}\label{eqapp}
\lambda(u)\approx   \lambda_{a,X}(u):=
\displaystyle{\sum_{j=1}^Ma_ju(x_j)   }\fa u\in U. 
\end{equation}
The background is that most operator equations can be written as infinitely may
linear equations
$$
\lambda(u)=f_\lambda \fa \lambda\in \Lambda\subset U^*,
$$
where the functionals evaluate weak or strong derivatives
or differential operators like the Laplacian or take boundary values.
This means that the classical approach of {\em meshless methods} 
is taken, namely to
write the approximations {\it entirely in terms of nodes} 
\cite{belytschko-et-al:1996-1}.

Our concern is to find {\em optimal} approximations in Sobolev space
$W_2^m(\Omega)$  for domains $\Omega\subset \R^d$. Their calculation 
is computationally costly and very unstable, but we shall prove that
there are {\em suboptimal}
approximations that can be calculated cheaply and stably,
namely via {\em scalable} approximations that 
have a certain exactness on
polynomials (Section \ref{SecOCiSS})
and may be 
constructed via polyharmonic kernels (Section \ref{SecPHS}). 
In particular, we shall show that they can have
the same convergence rate as the optimal
approximations, and we present the minimal
  assumptions on  the node sets to reach that optimal rate.

The application for all of this is that 
error bounds and convergence rates for nodal approximations 
to linear functionals enter into the consistency part
of the error analysis \cite{schaback:2016-1} of nodal meshless methods.
These occur in many papers in Science and Engineering, e.g.
\cite{aboiyar-et-al:2010-1,agarwal-basu:2012-1,%
  bayona-et-al:2012-1,chandhini-sanyarisanu:2007-1,%
  flyer-et-al:2015-1,flyer-et-al:2016-1,%
gerace-et-al:2009,hoang-et-al:2012-1,hosseini-hashemi:2011-1,iske:2013-1,%
sarler:2007-1,shankar-et-al:2015-1,shu-et-al:2003-1,shu-et-al:2005-1,%
stevens-et-al:2011-1,thai-et-al:2012-1,tolstykh:2000-1,%
vertnik-sarler:2011-1,yao-et-al:2011-1,yao-et-al:2012-1%
}, 
and several authors have analyzed the construction
of nodal approximations mathematically, e.g. 
\cite{davydov-oanh:2011-1,davydov-oanh:2011-2,
davydov-et-al:2016-1,iske:1995-1, iske:2003-1,
larsson-et-al:2013-1,wright-fornberg:2006-1},
but without considering optimal convergence rates.

To get started, we present a suitable notion of {\em scalability} 
in Section \ref{SecSca}
that allows to
define error functionals $\epsilon_h\in U^*$ 
based on the scaled point set $hX$ 
for small $h>0$ and 
to prove {\em convergence rates} $k$ in the sense that error bounds of the form
$\|\epsilon_h\|_{U^*}\leq Ch^k$ hold for $h\to 0$. The standard {\em derivative order} 
$|\alpha|$ of a pointwise multivariate derivative functional
$\lambda(u):=D^\alpha u(0)$ will reappear as a 
{\em scaling} order $s(\lambda)$ 
that governs how the approximations of a functional $\lambda$ 
scale for $h\to 0$.

Of course, optimal error bounds will crucially 
depend on the space $U$ and the node set $X$. If $U$ contains all 
real-valued polynomials, the achievable 
convergence rate of an approximation 
of a functional $\lambda$ based on a node set $X$ 
is limited by the maximal convergence rate
on the subspace of polynomials. Section \ref{SecEAoP} will prove that 
the upper limit of the convergence 
rate on polynomials is
$q_{max}(\lambda,X)-s(\lambda)$
where $q_{max}$ is the maximal order of polynomials on which the
  approximation is exact, 
and 
that
this rate can be reached by {\em scalable} approximations constructed
via exactness on polynomials.

But even if the node set $X$ is large enough to let 
approximations be exact on high-order polynomials, the convergence
rate may be restricted by limited smoothness of the functions in $U$. 
In Sobolev spaces $W_2^m(\R^d)$ or 
$W_2^m(\Omega)$ with $\Omega\subset\R^d$ the achievable rate 
for arbitrarily large node sets $X$ turns out to be
bounded above by $m-d/2-s(\lambda)$
in Section \ref{SecOCiSS}, so that 
\begin{equation}\label{eqOptRate}
\min\left(m-d/2-s(\lambda),q_{max}(\lambda,X)-s(\lambda)\right),
\end{equation}
is a general formula for an upper bound on the 
convergence rate in Sobolev space $W_2^m(\R^d)$, 
and this is confirmed by numerical experiments in Section \ref{SecExa}.

Then Sections  \ref{SecEAoP},  
  \ref{SecOCiSS}, and \ref{SecPHS} prove that the
  convergence rate \eref{eqOptRate} is {\em optimal}, 
  and it can be
achieved by  {\em scalable} 
stencils based solely on exactness on polynomials. 
Furthermore,
Section \ref{SecSC}  gives a sufficient condition
for the convergence of optimal stencils to scalable stencils.

A particularly interesting case is the {\em best compromise} case
where the two constraints on the convergence rate are equal, i.e.
\begin{equation}\label{eqcompro}
q_{max}(\lambda,X)=\lceil m-d/2\rceil.
\end{equation}
For a given smoothness $m$ it yields the sparsest approximation
that has the optimal
convergence rate (or comes arbitrarily close to it
if $m-d/2$ is an integer),
and for a given sparsity via $X$ it provides 
the minimal smoothness
that is required to realize the maximal possible rate of convergence using that
node set.

The numerical examples are collected in Section \ref{SecExa},
while the final section \ref{SecU} summarizes our results and
points out a few open problems for further research.  
\section{Scalability}\label{SecSca}
We now study the behavior of functionals and their approximations under
{\em scaling}.
\begin{definition}\label{defScaOrd}
\begin{enumerate}
 \item A domain $\Omega\subset\R^d$ is {\em scalable}, if it contains the
origin as an interior point and satisfies  $h\Omega\subseteq\Omega$ 
for all $0\leq h\leq 1$, i.e. if $\Omega$ is star--shaped
with repect to the origin.
\item
A space $U$ of functions on a scalable domain $\Omega$ 
is {\em scalable}, if $u(h\cdot)$ is in $U$ for all $0<h\leq 1$ 
and all $u\in U$. 
\item
A functional $\lambda\in U^*$ on a scalable space $U$ 
has  {\em scaling order} or {\em homogeneity order} $s$ if
$$
\lambda(u(h\cdot))
=h^{s}\lambda(u) \fa u\in U.
$$ 
\end{enumerate} 
\end{definition}
Of course, this means that the functional $\lambda$ must be 
local in or near the origin.  
For example, the standard strong functionals are 
modelled by multivariate derivatives 
$$
\lambda_\alpha(u)=\dfrac{\partial^\alpha u}{\partial x^\alpha}(0)
$$
at zero, with the scaling behaviour
$$
\begin{array}{rcl}
\lambda_\alpha(u(h\cdot))
&=&h^{|\alpha|}\lambda_\alpha(u)
\end{array}
$$
showing that the scaling order coincides with the order of differentiation here.
This generalizes to all linear homogeneous differential 
operators, e.g. the Laplacian.

Having dealt with scalability of $\lambda$, we now turn to scalability
of the nodal approximation $\lambda_{a,X}$ of \eref{eqapp}.
To match the scalability order $s$ of $\lambda$, we should assume the same
$h$ power for  $\lambda_{a,X}$, and consider
$$
h^{-s}\lambda_{a,X}(u(h\,\cdot))=\displaystyle{\sum_{j=1}^Ma_jh^{-s}u(hx_j)}
= \lambda_{ah^{-s},hX}(u)
$$
for all $u\in U$ and $0<h\leq 1$. This is the right notion
of scalability for the approximation, but now
we need the $h$ dependence and refrain from setting this equal to
$\lambda_{a,X}(u)$ like in Definition \ref{defScaOrd}. 
\begin{definition}\label{defErrScaOrd}
\begin{enumerate}
 \item
An approximation \eref{eqapp} to a scalable functional $\lambda$
of scaling order $s$ is {\em scalable} of the same order, if the error
functional is scalable of order $s$, i.e.
\begin{equation}\label{eqerrfunct}
\epsilon_h(u):=\lambda(u)-\lambda_{ah^{-s},hX}(u)
=h^{-s}(\lambda-\lambda_{a,X})(u(h\,\cdot))=h^{-s}\epsilon_1(u(h\cdot)) 
\end{equation}
for all $u\in U,\;0<h\leq 1$. 
\item A scalable approximation \eref{eqerrfunct} will be called a {\em stencil}.
\item If an approximation \eref{eqapp} is given for $h=1$, and if
the functional $\lambda$ has scaling order $s$,   the transition
to \eref{eqerrfunct} by using weights $a_jh^{-s}$ in the scaled case will
be called {\em enforced} scaling. 
\end{enumerate} 
\end{definition} 
A standard example is the five-point star approximation
$$
-\Delta u(0,0)\approx \dfrac{1}{h^2}(4u(0,0)-u(0,h)-u(0,-h)-u(h,0)-u(-h,0))
$$
to the Laplacian in 2D, and all other notions of generalized divided differences
that apply to scaled node sets $hX$.

The scaled form in \eref{eqerrfunct} allows the very simple error bound
$$
|\epsilon_h(u)|\leq h^{-s}\|\lambda-
 \lambda_{a,X}\|_{U^*}\|u(h\cdot)\|_{U} \fa u\in U
$$
that is useful if 
$\|u(h\cdot)\|_{U}$ is accessible and behaves nicely for $h\to 0$.

Weights of scalable approximations can be calculated at large scales and then scaled down
by multiplication. This bypasses instabilities for small $h$ and saves a lot of
computational work, in particular if applications work on multiple scales
or if meshless methods use the same geometric pattern of nodes repeatedly,
e.g. in Meshless Local Petrov Galerkin \cite{atluri:2005-1} techniques.

However, {\em optimal} approximations in Sobolev spaces will not be scalable. 
This is why the rest of the paper studies 
how close scalable approximations come to the optimal
ones analyzed in \cite{davydov-schaback:2016-2}.
\section{Optimal Convergence on Polynomials}\label{SecEAoP}
We first relate the approximation error of nodal approximations 
to exactness on polynomials and assume that a scalable functional $\lambda$ of 
scaling order $s$ is  given that is applicable to all $d$-variate 
polynomials. This will be true, for instance, in all Sobolev spaces
$W_2^m(\Omega)$ for bounded scalable domains $\Omega\subset\R^d$. 
The space of all real-valued $d$-variate polynomials up to order $q$
will be denoted by $\Pqd$, and for a given node set $X\subset\R^d$ 
and a functional $\lambda$ we define
$$
q_{max}(\lambda,X)=\max\{q\;:\; \lambda- \lambda_{a,X}=0 
\hbox{ on } \Pqd \hbox{ for some } a \in\R^{|X|}\}
$$
to be the maximal possible {\em polynomial exactness order} (abbreviated by
PEO in the figures of the examples)  of a 
nodal approximation \eref{eqapp} to $\lambda$ based on $X$. 
\begin{theorem}\label{theqsbnd}
  Consider a fixed set $X\subset\R^d$ and
  a functional $\lambda$.
  If a sequence of general nodal approximations
  $\lambda_{a(h),hX}$ converges to $\lambda$ on a
  space spanned by finitely many monomials, then $X$ admits an approximation
  to $\lambda$ that is exact on these monomials.
\end{theorem} 
\begin{proof}
Due to 
\begin{equation}\label{eqpolscale}
\lambda(x^\alpha)- \lambda_{a(h),hX}(x^\alpha)
=
\lambda(x^\alpha)- \lambda_{a(h)h^{|\alpha|},X}(x^\alpha),
\end{equation}
convergence of functionals $\lambda_{a(h),hX}$
to $\lambda$ on a set of monomials
implies that  the error of the best approximation
to $\lambda$ by functionals $\lambda_{a,X}$, restricted
to the space spanned by those monomials, is zero.
\end{proof}
We now know an upper bound for the
maximal order of polynomials for which
approximations can be convergent,
if $X$ and $\lambda$ are fixed.
This order can be achieved for scalable stencils:
\begin{theorem}\label{thePolOrd}
If all polynomials are in $U$, the convergence rate of 
a scalable stencil of scaling order $s$ based on a point set $X$ 
on all polynomials is exactly $q_{max}(\lambda,X)-s$ if the stencil is 
exact on $\Pqd$ for $q=q_{max}(\lambda,X)$. The convergence rate
on all of $U$ is bounded above by $q_{max}(\lambda,X)-s$.
\end{theorem} 
\begin{proof}
  We apply \eref{eqpolscale} in the scalable situation and get
  $$
h^{-s}\lambda((hx)^\alpha)- h^{-s} \lambda_{a,X}((hx)^\alpha)
=
h^{-s+|\alpha|}\left(\lambda(x^\alpha)- \lambda_{a,X}(x^\alpha)\right),
$$
proving the assertion.
\end{proof}
Consequently, if a node set $X=\{x_1,\ldots,x_M\}$ is given, 
if the application allows all polynomials,
and if one wants a scalable stencil, 
the best one can
do is to take a stencil with maximal order $q_{max}(\lambda,X)$
of polynomial exactness. It will lead to
a scalable stencil with the optimal convergence rate 
among all approximations.
Additional tricks cannot improve that rate, but it can be smaller 
due to restricted smoothness of functions in $U$. This will be the topic
of Section \ref{SecOCiSS}.

If exactness of order $q$ is required in applications, one takes a basis
$p_1,\ldots,p_Q$ of the space $\P_q^d$ 
of $d$-variate polynomials of order $q$ with $Q=\dim \P_q^d={q+d-1\choose d}$
and has to find a solution of the linear system
\begin{equation}\label{eqlinsys}
\lambda(p_k)=\displaystyle{\sum_{j=1}^Ma_jp_k(x_j),\;1\leq k\leq Q   }. 
\end{equation}
This may exist even in case $M<Q$, the simplest example being
the five-point star in 2D 
for $\lambda(u)=\Delta u(0)$ which is exact of order 4, while 
$M=5<Q=10$. For general point sets, there is no way around 
setting up and solving the above linear system.

If the system has a solution, we get a stencil by enforced scaling and
with error 
$$
h^{-s}\lambda(u(h\cdot))-h^{-s}\displaystyle{\sum_{j=1}^Ma_ju(hx_j)}
$$
which then is polynomially exact of order $q$ and has convergence rate
$k=q-s$, 
but only on polynomials.
If $U$ contains functions of limited smoothness, 
this convergence rate will not be attained for all functions in $U$.
We shall prove
in Section \ref{SecOCiSS} that the convergence rate in $W_2^m(\Omega)$
for $\Omega\subseteq\R^d$ is limited by $m-s-d/2$, no matter how large
the order $q$ of polynomial exactness on $X$ is.

To make this construction partially independent of the functionals,
we add
\begin{definition}\label{defqX}
A  finite point set $X=\{x_1,\ldots,x_M\}\subset\R^d$  has {\em polynomial
reproduction} of order $q$, if all polynomials in $\P_q^d$ can be recovered
from their values on $X$. 
\end{definition} 
\begin{theorem}\label{theExOnPolRep}
If the set $X$ allows polynomial reproduction of order $q$, then all
admissible linear functionals of scaling order $s\leq q$ 
have a stencil that is exact at least of order $q$,
by applying $\lambda$ to a Lagrange basis
of $\P_q^d$. This stencil has convergence rate 
at least $q-s$ on polynomials. 
\end{theorem} 
\begin{proof}
Let the set $X$ allow polynomial reproduction of order $q$. Then, for
$Q=\dim \P_q^d$, there are 
polynomials $p_1,\ldots,p_Q$ and a subset $Y=\{y_1,\ldots,y_Q\}\subseteq X$ 
such that the representation 
$$
p(x)=\displaystyle{\sum_{j=1}^Qp(y_j)p_j(x)\fa p\in \P_q^d   } 
$$  
holds, and the matrix of values $p_k(y_j),\;1\leq j.k\leq Q$ is the identity. 
This implies $Q\leq M$, and the stencil satisfying
$$
\lambda(p)=\displaystyle{\sum_{j=1}^Qp(y_j)\lambda(p_j)\fa p\in \P_q^d   } 
$$
with weights $a_j:=\lambda(p_j)$ is exact on $\P_q^d$. The rest follows like
above. 
\end{proof}
But note that the five-point star is an example of an approximation on a set
that has polynomial reproduction only of order $2$, while it 
has a scalable stencil for the Laplacian that is exact on polynomials
of order up to $4$ and convergent of rate 2. The application
of Theorem \ref{theExOnPolRep} would require polynomial reproduction 
of order $4$ for the same convergence rate.

In general, one can use the $M$ given nodes for getting exactness 
on polynomials of maximal order, and then there can be additional degrees of
freedom because the $Q\times M$ linear system \eref{eqlinsys}
may be nonuniquely solvable. 
The paper \cite{davydov-schaback:2016-1} deals with various techniques
to use the additional degrees of freedom, e.g. for minimizing the $\ell_1$ 
norm of the weights. In all cases the result is scalable and then 
this paper applies as well. On the other hand, the paper 
\cite{davydov-schaback:2016-2} focuses on non-scalable
approximations induced by kernels. Both papers perform their 
convergence analysis mainly for single approximations.
While 
this paper focuses on convergence rates in Sobolev spaces,
\cite{davydov-schaback:2016-2} considers Hölder spaces and Sobolev spaces
$W_\infty^r$. 
A third way to use additional degrees of freedom is to
take optimal stencils for polyharmonic kernels in 
Beppo-Levi spaces, see Section 
\ref{SecPHS}.

But before we go over from polynomials 
to these spaces, we remark that many application papers
use meshless methods to solve problems that have true solutions   
$u^*$ with rapidly convergent power series representations
(see e.g \cite{kansa:2015-1} for a recent example 
with $u^*(x,y)=\exp(ax+by)$).
In such cases, a high order of polynomial exactness pays off, but as soon as 
the problem is treated in Sobolev space, this advantage is gone. 
A truly worst-case analysis of nodal meshless methods is in 
\cite{schaback:2016-1}.

This discussion showed that on polynomials one can get
stencils of arbitrarily high convergence rates, provided that there are enough
nodes to ensure exactness on high-degree polynomials. 
For working on spaces of functions with limited smoothness, the latter
will limit the convergence rate of the stencil, and we want to show how.
\section{Optimal Convergence in Sobolev Spaces}\label{SecOCiSS}
Our goal is to reach the optimal convergence rates in Sobolev spaces
via cheap, scalable,  and stable stencils, and for this we need to know those
rates. But before that, we want to eliminate the difference between local and
global Sobolev spaces, as far as convergence {\em rates} are concerned.

Local Sobolev functionals
are global ones due to 
$W_2^m(\Omega)^*\subset W_2^m(\R^d)^*$ that follows from 
$W_2^m(\Omega)\supset W_2^m(\R^d)$ for Lipschitz domains.
This implies that we can evaluate
the norm of each functional $\lambda\in W_2^m(\Omega)^*$ 
in $W_2^m(\R^d)^*$ via the kernel, up to a fixed multiplicative 
constant.

For the other way round and in the scalable case, 
we consider the subspace $L_\Omega$ 
of all point-based functionals $  \lambda_{a,X} \in W_2^m(\R^d)^*$ 
with sets $X\subset\Omega$ and $a\in\R^{|X|}$ for a scalable domain 
$\Omega\subset\R^d$ and form its closure $\call_\Omega$ under the  
kernel-based $W_2^m(\R^d)^*$ norm. Exactly these functionals are those that we
study here. Since the spaces $W_2^m(\R^d)$ and $W_2^m(\Omega)$ 
are norm-equivalent, the limit process is the same in $W_2^m(\Omega)$,
and therefore we have that  $\call_\Omega\subset W_2^m(\Omega)^*$.
\begin{theorem}\label{theCREqui}
The functionals considered here are always in the space 
$\call_\Omega\subset W_2^m(\Omega)^*$, and their norm can be evaluated in
$W_2^m(\R^d)^*$ up to a space- and domain- dependent constant.
The convergence rates in $W_2^m(\Omega)^*$ and $W_2^m(\R^d)^*$ are the same.\qed
\end{theorem}  
In Section \ref{SecPHS} we shall extend this argument to Beppo-Levi spaces.
\begin{theorem}\label{theopt}
The convergence rate of any nodal approximation to a scalable functional
$\lambda$  of scalability order $s$ on  
$W_2^m(\R^d)$ with $m>d/2$
is at most $m-s-d/2$. 
\end{theorem}
\begin{proof}
We need at least $m>d/2$ 
to let the nodal approximations $\lambda_{a,X}$
of \eref{eqapp} to be well-defined. Then we take a ``bump'' 
function $v\in W_2^m(\R^d)$ that vanishes on $X$ and has
$\lambda(v)\neq 0$.

Now we scale and consider $\lambda_{a(h),hX}$ as an approximation
on $hX$ with error functional
$$
\epsilon_h=\lambda-\lambda_{a(h),hX}.
$$
Then
$$
\begin{array}{rcl}
\epsilon_h(v(\cdot/h))
&=& \lambda(v(\cdot/h))-\lambda_{a(h),hX}(v(\cdot/h))\\
&=& h^{-s}\lambda(v)-0\\
\end{array}
$$
and
$$
\begin{array}{rcl}
\|v(\cdot/h)\|^2_{W_2^m(\R^d)}
&=&
\displaystyle{\sum_{|\alpha|\leq m}\int_{\R^d}|D^\alpha(v(\cdot/h))|^2}\\ 
&=&
\displaystyle{\sum_{|\alpha|\leq m}h^{-2|\alpha|}\int_{\R^d}|D^\alpha(v)(x/h)|^2dx}\\ 
&=&
\displaystyle{h^d\sum_{|\alpha|\leq m}h^{-2|\alpha|}\int_{\R^d}|D^\alpha(v)(y)|^2dy}\\ 
&\leq&
\displaystyle{h^{d-2m}\|v\|^2_{W_2^m(\R^d)}}\\ 
\end{array}
$$
leading to 
$$
\begin{array}{rcl}
\|\epsilon_h\|_{W_2^m(\R^d)^*}
&=&
\displaystyle{ \sup_{u\in W_2^m(\R^d)\setminus \{0\}}
\dfrac{|\epsilon_h(u)|}{\|u\|_{W_2^m(\R^d)}}}\\[0.5cm]
&\geq&
\displaystyle{\dfrac{|\epsilon_h(v(\cdot/h))|}{\|v(\cdot/h)\|_{W_2^m(\R^d)}}}\\[0.5cm]
&\geq&h^{-s}
\displaystyle{\dfrac{|\lambda(v)|}{\|v(\cdot/h)\|_{W_2^m(\R^d)}}}\\[0.5cm]
&\geq&h^{m-s-d/2}
\displaystyle{\dfrac{|\lambda(v)|}{\|v\|_{W_2^m(\R^d)}}}.
\end{array} 
$$
\end{proof}
This holds for all weights, including the non-scalable optimal ones,
and for all nodal point sets $X$.

Our next goal is to show that this rate is attainable 
for scalable stencils with sufficient polynomial exactness,
in particular for optimal stencils calculated via polyharmonic kernels. 
\begin{theorem}\label{theSobbound}
  Let $\lambda$ be a functional of scaling order $s$
  that is continuous on $W_2^\mu(\Omega)$ for some $\mu>d/2$, and
let $X$  allow a polynomially exact approximation to $\lambda$ of  
of some order $q\geq \mu>d/2$. 
Then any scalable 
stencil for approximation of $\lambda$ 
on $X$ with that exactness
has the optimal convergence rate
$m-s-d/2$ in $W_2^m(\Omega)$ for all $m$ with $\mu\leq m < q+d/2$.
In case $m=q+d/2$, the rate is at least $m-s-d/2-\epsilon=q-s-\epsilon$
for arbitrarily small $\epsilon >0$.
\end{theorem}  
\begin{proof}
We first treat the case $m\leq q$.
By the Bramble-Hilbert lemma \cite{bramble-hilbert:1970-1}, 
the error functional defined by
$$
\epsilon(u)=\lambda(u)-  \lambda_{a,X}(u)
$$ 
is continuous on $W_2^m(\Omega)$
  and vanishes on $\P_m^d$. Then it has an error
bound
$$
|\epsilon(u)|
\leq \|\epsilon\|_{W_2^m(\Omega)^*}|u|_{W_2^m(\Omega)}\fa u\in W_2^m(\Omega).
$$ 
This leads to
$$
\begin{array}{rcl}
|h^{-s}\lambda(u(h\cdot))- h^{-s} \lambda_{a,X}(u(h\cdot))|
&=&
h^{-s}|\epsilon(u(h\cdot))|\\
&\leq &
h^{-s}\|\epsilon\|_{W_2^m(\Omega)^*}|u(h\cdot)|_{W_2^m(\Omega)}\\
&=& 
h^{-s}\|\epsilon\|_{W_2^m(\Omega)^*}h^{m-d/2}|u|_{W_2^m(h\Omega)}\\
&\leq& 
h^{-s}\|\epsilon\|_{W_2^m(\Omega)^*}h^{m-d/2}|u|_{W_2^m(\Omega)}
\end{array}
$$ 
where we used
\begin{equation}\label{equhhu} 
\begin{array}{rcl}
|u(h\cdot)|^2_{W_2^m(\Omega)}
&=& 
\displaystyle{\sum_{|\alpha|=m}\int_\Omega
\left|D^\alpha (u(h\cdot))(x)  \right|^2dx}\\
&=& 
\displaystyle{h^{2m}\sum_{|\alpha|=m}\int_\Omega
\left|D^\alpha (u)(hx)  \right|^2dx}\\
&=& 
\displaystyle{h^{2m-d}\sum_{|\alpha|=m}\int_{h\Omega}
\left|D^\alpha (u)(y)  \right|^2dy}\\
&=&
h^{2m-d}|u|^2_{W_2^m(h\Omega)}.
\end{array}
\end{equation}
For the case $q\leq m <q+d/2$ we repeat the argument, but now
in $W_p^q(\Omega)\supseteq W_2^m(\Omega)$
for $p\in [2,\infty)$ with $q-d/p=m-d/2$.
  Because of $q\geq \mu$ we also have $W_p^q(\Omega)\subseteq W_2^\mu(\Omega)$,
  guaranteeing continuity on $W_p^q(\Omega)$.
The corresponding proof steps are
$$
  \begin{array}{rcl}
    |h^{-s}\lambda(u(h\cdot))- h^{-s} \lambda_{a,X}(u(h\cdot))|
    &\leq &
    h^{-s}\|\epsilon\|_{W_p^q(\Omega)^*}h^{q-d/p}|u|_{W_p^q(\Omega)},\\
    |u(h\cdot)|^p_{W_p^q(\Omega)}
    &=& h^{pq-d}|u|^p_{W_p^q(\Omega)}.
\end{array}
$$
  For $m=q+d/2$, the space $W_2^m(\Omega)$ is embedded in $W_p^q(\Omega)$
  for arbitrary
  $p\in [2,\infty)$, and on that space we get
    the rate $q-s-d/p=m-s-d/2-d/p$.
\end{proof}
Theorem \ref{theSobbound}
proves optimality of the convergence rate \eref{eqOptRate},
  and it shows that the optimal rate is attained
  by {\em scalable} stencils whose point sets allow
  polynomial exactness of some order larger than $m-d/2$.
  
  In view of the {\em best compromise} situation,
  one can ask for the minimal polynomial exactness order
  $q$ that allows the optimal convergence rate
  for fixed $m$ and $d$.
    If $m-d/2$ is not an integer,
  this is $q:=\lceil m-d/2\rceil$ as in \eref{eqcompro}.
  In the exceptional case $m-d/2\in \N$,
  the order $m-d/2+1$ is sufficient for the optimal rate,
  but order $m-d/2$ can come arbitrarily close to it.
  We shall deal with this situation 
in Sections \ref{SecPHS} and  \ref{SecExa}.

Consequently, large orders of polynomial exactness will not pay off,
if smoothness is the limiting factor. 
If the size of the  point set $X$ is the limiting factor, we get 
\begin{corollary}\label{cortheCbound}
Let $\lambda$ be a functional of scaling order $s$ which is continuous
on $W_2^\mu(\Omega)$ with integer $\mu>d/2$, and 
let $X$  allow a polynomially exact approximation to $\lambda$ of  
of some order $q\geq \mu$. Then any scalable 
stencil for approximation of $\lambda$ 
on $X$ with that exactness
has convergence rate at least 
 $q-s$ in $W_2^m(\Omega)$ for all $m>q+d/2$.
\end{corollary} 
\begin{proof}
  We repeat the proof of Theorem \ref{theSobbound},
  but now on $W_2^q(\Omega)$ and get
$$
\begin{array}{rcl}
|h^{-s}\lambda(u(h\cdot))- h^{-s} \lambda_{a,X}(u(h\cdot))|
&=&
h^{-s}|\epsilon(u(h\cdot))|\\
&\leq &
h^{-s}\|\epsilon\|_{W_2^q(\Omega)^*}|u(h\cdot)|_{W_2^q(\Omega)}.
\end{array}
$$ 
Then we use \eref{equhhu} replacing $m$ by $q$ there,
but insert functions $u\in W_2^m(\Omega)$ for $m>q+d/2$.
Then the $q$-th derivatives in \eref{equhhu}
will be continuous, proving
$$
\begin{array}{rcl}
|u(h\cdot)|^2_{W_2^q(\Omega)}
&=& 
\displaystyle{h^{2q}\sum_{|\alpha|=q}\int_\Omega
\left|D^\alpha (u)(hx)  \right|^2dx}\\
&\leq &C 
h^{2q}\|u\|_{C^q(\Omega)}.
\end{array}
$$
Thus the convergence rate
in $W_2^m(\Omega)$ 
is at least $q-s$.
\end{proof}
This argument used continuity
of higher derivatives to bound local integrals, as in 
\cite{davydov-schaback:2016-2}.

Note that
Corollary \ref{cortheCbound} produces only integer 
or half-integer convergence rates
while Theorem \ref{theSobbound} allows general non-integer rates.
We shall give examples in Section \ref{SecExa}.

To summarize, we get
convergence rates for scalable stencils as in Table \ref{tabSobRates}. 
For the case in the second row, the optimal convergence behavior
  is not reached for order $q$, but for order $q+1$ by applying the first
  row. For given $m$ and $d$, a scalable stencil
  with polynomial exactness
  order $\lfloor m-d/2\rfloor+1$ is sufficient for optimal convergence
  in $W_2^m(\Omega),\;\Omega\subset\R^d.$ By solving the system
  \eref{eqlinsys}, such stencils are easy to calculate, but
if the system is underdetermined, one should make good use of the additional
degrees of freedom. This topic is treated in \cite{davydov-schaback:2016-1}
by applying optimization techniques, while the next sections will 
focus on unique stencils obtained by polyharmonic kernels.
Because the latter come close to the kernels reproducing Sobolev spaces,
they should provide good approximations to the non-scalable optimal
approximations in Sobolev spaces. 
\begin{table}[hbt]\centering
\begin{tabular}{||c||c|c|}\hline
$m$ and $q$  & minimal rate  & optimal rate\\\hline
$m <q+d/2$ & $m-s-d/2$ & yes \\
  $m=q+d/2$ & $m-s-d/2-\epsilon,\;\epsilon>0 $ & no, 
  $m-s-d/2=q-s$ \\
$m>q+d/2$ & $q-s$& yes for  $q= q_{max}(\lambda,X)$\\
\hline
\end{tabular}
\caption{Convergence rates in $W_2^m(\R^d)$ for scalable stencils 
  defined on $W_2^\mu(\R^d)$ with polynomial exactness $q\geq \mu>d/2$.
  \label{tabSobRates}}
\end{table}

\section{Polyharmonic Kernels}\label{SecPHS}
For $m-d/2>0$ real, we define 
the polyharmonic kernel   
\begin{equation}\label{eqKmd}
H_{m,d}(r):=(-1)^{\lfloor m-d/2 \rfloor +1}
\left\{
\begin{array}{ll}
r^{2m-d}\log r, &2m-d \hbox{ even integer }\\ 
r^{2m-d},& \hbox{ else }
\end{array} 
\right\}
\end{equation}
up to a positive scalar multiple. 
This kernel is conditionally positive definite
of order 
$$
q(m-d/2):=\lfloor m-d/2 \rfloor +1.
$$
%

For comparison, the Whittle-Mat\'ern kernel generating 
Sobolev space $W_2^m(\R^d)$ is, up to a positive constant,
$$
S_{m,d}(r):=
K_{m-d/2}(r)r^{m-d/2}
$$
with the modified Bessel function of second kind.
The generalized $d$-variate 
Fourier transforms then are 
$$
\begin{array}{rcl}
\hat H_{m,d}(\omega)&=&\|\omega\|_2^{-2m},\\
\hat S_{m,d}(\omega)&=&(1+\|\omega\|_2^2)^{-m},
\end{array} 
$$
up to positive constants,
showing a similarity that we will not explore further at this point. 

While $S_{m,d}$ reproduces $W_2^m(\R^d)$, the polyharmonic kernel $H_{m,d}$
reproduces the {\em Beppo-Levi} space $BL_{m,d}$. 
This has a long history, see e.g. 
\cite{iske:1995-1, schaback:1997-3,iske:2003-1,%
wendland:2005-1,beatson-et-al:2005-1,iske:2011-1},
but we take a shortcut here and refer the reader to the background literature.
From the paper \cite{iske:2003-1} of A. Iske 
we take the very useful fact that optimal approximations in Beppo-Levi spaces 
using polyharmonic kernels are always scalable and can be stably and
efficiently calculated. We shall investigate the optimal convergence rate 
in Sobolev and Beppo-Levi space here, while \cite{iske:2003-1} contains
convergence rates in $C^m(\Omega)$.

A typical scale-invariance property of Beppo-Levi spaces is 
\begin{equation}\label{equhBL}
\|u(h\cdot)\|_{BL_{m,d}}=h^{m-d/2}\|u\|_{BL_{m,d}}  \fa u \in BL_{m,d}.
\end{equation}
Note the similarity between the above
formula and \eref{equhhu} used the proof of Theorem \ref{theSobbound},
because the classical 
$W_2^m(\R^d)$ seminorm coincides with the norm in $BL_{m,d}$.
\begin{theorem}\label{theGenPHSOpt}
Let 
a scalable approximation \eref{eqapp} of scaling order $s$  
be exact on the polynomials of some order 
$q\geq q(m-d/2)=\lfloor m-d/2 \rfloor +1$ and assume that $\lambda- \lambda_{a,X}$
is in $BL_{m,d}^*$. 
Then this stencil has the exact convergence rate $m-s-d/2$ in $BL_{m,d}$.
\end{theorem} 
\begin{proof}
We evaluate the norm of the error functional after scaling via
$$
\begin{array}{rcl}
\|\lambda-h^{-s} \lambda_{a,hX}\|_{BL_{m,d}^*}
&=&\displaystyle{   \sup_{\|u\|_{BL_{m,d}}\leq 1}|\lambda(u)-h^{-s} \lambda_{a,hX}(u)|}\\
&=&h^{-s}\displaystyle{   
\sup_{\|u\|_{BL_{m,d}}\leq 1}|\lambda(u(h\cdot))- \lambda_{a,X}(u(h\cdot))|}\\
&=&h^{-s+m-d/2}\displaystyle{   
\sup_{\|u(h\cdot)\|_{BL_{m,d}}\leq 1}|\lambda(u(h\cdot))- \lambda_{a,X}(u(h\cdot))|}\\
&=&h^{-s+m-d/2}\|\lambda- \lambda_{a,X}\|_{BL_{m,d}^*}
\end{array} 
$$
using that \eref{equhBL} implies that the unit balls of all $u$ and all $u(h\cdot)$
are the same up to a factor.
\end{proof}
\begin{corollary}\label{corPHS1}
Polynomial exactness of more than order $\lfloor m-d/2 \rfloor +1$ does not pay
off
in a higher convergence rate in Beppo-Levi space $BL_{m,d}$. \qed
\end{corollary} 
\begin{corollary}\label{corPHS2}
Let a point set $X=\{x_1,\ldots,x_M\}\subset\Omega\subset\R^d$  
be given such that
there is some approximation \eref{eqapp} that is exact on polynomials
of order  $\lfloor m-d/2 \rfloor +1$ and that has 
$\lambda- \lambda_{a,X}\in BL_{m,d}^*$. 
Then there is a weight vector $a^*\in\R^M$ that minimizes
$\|\lambda- \lambda_{a,X}\|_{BL_{m,d}^*}$ under all competing approximations, 
and the resulting stencil is $BL_{m,d}$-optimal under all stencils of at least that 
polynomial exactness. \qed
\end{corollary}
By applying Theorem \ref{theSobbound}, we get
  \begin{corollary}\label{corPH4Scal}
  One can use optimal scalable stencils
  obtained via polyharmonic kernels $H_{m,d}$ to get optimal convergence rates
  in $W_2^m(\Omega)$ for $\Omega\subset\R^d$, 
  provided that the underlying sets allow exactness
  on polynomials of order
  $q(m-d/2)=\lfloor m-d/2\rfloor+1$.\qed
  \end{corollary}
  If $m-d/2$ is not an integer, the above order
is smallest possible
for optimal convergence. For $m-d/2$ integer, we have
$$
q(m-d/2)=\lfloor m-d/2\rfloor+1=m-d/2+1,
$$
and Theorem \ref{theSobbound}
suggests that we could come arbitrarily close to the optimal
convergence rate if we use order $q=m-d/2$. But then we cannot use
the polyharmonic kernel $H_{m,d}$.

However, there is a workaround. We construct a scalable stencil
via the polyharmonic kernel   $H_{m',d}$ for $m-1\leq m'<m$
using polynomial exactness of order $q(m'-d/2)=q$.
By Theorem \ref{theSobbound} this yields a convergence rate
  at least $m-s-d/2-\epsilon$ for all $\epsilon >0$, no matter how
  $m'$ was chosen.
\begin{corollary}\label{corSecCase2}
For the special situation $m=q+d/2$ in Table 
\ref{tabSobRates} there is a scalable stencil
with
polynomial exactness order $q$,
based on a polyharmonic kernel,  that has convergence rate
at least $m-s-d/2-\epsilon$
for all $\epsilon>0$. \qed
\end{corollary}
\section{Stable Error Evaluation}\label{SecStEE}
In the most interesting cases, the leading term of the 
error of a scalable stencil
in Sobolev space can be stably calculated via polyharmonic kernels.
To prove this, we show now that the 
polyharmonic kernels $H_{m,d}$ arise naturally 
as part of the kernels $S_{m,d}$ reproducing Sobolev space $H^m(\R^d)$. 
The latter 
have expansions as series in $r$, beginning with a finite number of even powers
with alternating signs. Such even powers, when written as
$r^{2k}=\|x-y\|_2^{2k}$ are polynomials in $x$ and $y$. After these even powers,
the next term
is a polyharmonic kernel:
\begin{theorem}\label{theExp}
The first non-even term in the expansion of $\sqrt{\frac{2}{\pi}}
K_{n+1/2}(r)r^{n+1/2}$ into powers of $r$ for integer $n\geq 0$
is the polyharmonic kernel 
$$
r^{2n+1}\dfrac{(-1)^{n+1}}{(2n+1)(2n-1)(2n-3)\cdots 1}= 
r^{2n+1}\dfrac{(-1)^{n+1}2^n\,n!}{(2n+1)!}.
$$
The first non-even term in the expansion
of $K_n(r)r^n$ for integer $n\geq 0$ is the polyharmonic kernel
$(-1)^{n+1}r^{2n}\log(r)\frac{2^{-n}}{n!}$.
\end{theorem} 
\begin{proof}
Equation 10.39.2 of \cite{NIST:2015-1} has $n=0$ of 
$$
\sqrt{\frac{2}{\pi}}
K_{n+1/2}(r)r^{n+1/2}=q_n(r)=e^{-r}p_n(r)
$$
with a polynomial $p_n$ of degree at most $n, \;p_0(r)=1,\;q_0(r)=e^{-r}$.
It can easily be shown that $rp_{n-1}(r)+p_n'(r)=p_n(r)$ holds,
using the derivative of the above expression, and similarly 
one gets
$$
-rq_{n-1}(r)= q_n'(r)
$$
from that derivative formula.
If we make it explicit
by
$$
q_n(r)=:\displaystyle{\sum_{j=0}^\infty q_{j,n}r^j   },
$$
we get
$$
\begin{array}{rcl}
-q_{k-1,n-1}
&=& q_{k+1,n}(k+1),\, k,n\geq 1\\ 
0&=&q_{1,n},\;n\geq 1.\\
\end{array}
$$
The assertion $q_{2k-1,n}=0$ for $1\leq k\leq n$ is
true for $k=1$ and all $n\geq 1$. Assume it to be 
true for $k$ and all $n\geq k$.
Then for all $n\geq k\geq 1$, 
$$
\begin{array}{rcl}
0=-q_{2k-1,n}
&=& q_{2k+1,n+1}(2k+1),\, 2k\geq 1, n\geq 0\\ 
\end{array}
$$
proves the assertion.
The first odd term of the kernel expansion
is $q_{2n+1,n}r^{2n+1}$, and its
coefficient  has the recursion
$$
\begin{array}{rcl}
-q_{2n-1,n-1}
&=& q_{2n+1,n}(2n+1),\, n\geq 1.\\ 
\end{array}
$$
For the other case we use equation (10.31.1) of \cite{NIST:2015-1} 
in shortened form as 
$$
K_n(z)z^n=p_n(z^2)+(-1)^{n+1}z^n\log(z/2)I_n(z)
$$
with an even power series $p_n(z^2)$, and due to (10.25.2) of 
\cite{NIST:2015-1} 
we have $I_n(z)=z^nq_n(z^2)$ with an even power series $q_n(z^2)$
with $q_n(0)=\frac{2^{-n}}{n!}$. Thus
$$
K_n(z)z^n=p_n(z^2)+(-1)^{n+1}z^{2n}\log(z/2)q_n(z^2),
$$
and the first non-even term of the expansion of $K_n(r)r^n$ is
the polyharmonic kernel
$$
(-1)^{n+1}r^{2n}\log(r)q_n(0)=(-1)^{n+1}r^{2n}\log(r)\frac{2^{-n}}{n!}.
$$
\end{proof}
We now are ready to show that
a good approximation of the error in Sobolev space can
be calculated stably via the error in Beppo-Levi space,
i.e. via polyharmonic kernels: 
\begin{theorem}\label{theComp}
Assume a scalable stencil of scalability order $s$ 
on a set $X\subset\R^d$ 
to be given with polynomial exactness $q$. For 
all integer $m$ with $\lfloor m-d/2\rfloor+1\leq q$, its error norm can be
evaluated on all Beppo-Levi 
spaces $BL_{m,d}$ and on
Sobolev space $W_2^m(\R^d)$. The convergence rate in both cases 
then is
$m-s-d/2$, and the quotient of errors converges to 1 for $h\to 0$,
if the scalar factors in the Sobolev and polyharmonic kernel
are aligned properly, namely as given in Theorem \ref{theExp}.
\end{theorem} 
\begin{proof}\label{ProtheComp}
The squared norm of the stencil's error functional can be evaluated 
on Sobolev space $W_2^m(\R^d)$ by 
$$
\begin{array}{rcl}
&&
\epsilon(h)^x\epsilon(h)^yK(x,y)\\
&=&
\displaystyle{   h^{-2s}\left(\lambda^x\lambda^yK(hx,hy) 
-2\sum_{j=1}^Ma_j\lambda^yK(hx_j,hy)   \right.}\\
&&
+\displaystyle{  \left.
-2\sum_{j,k=1}^Ma_ja_k\lambda^yK(hx_j,hx_k)   \right)}\\
\end{array} 
$$
where we used $K(x,y)$ as a shortcut for $K_{m-d/2}(\|x-y\|_2)\|x-y\|_2^{m-d/2}$
and ignore scalar multiples. 
Now we insert the series expansions of Theorem \ref{theExp}.
For odd $d$ and $m-d/2=n+1/2$ we
have, up to constant factors,
$$
K_{m-d/2}(r)r^{m-d/2}=\displaystyle{\sum_{j=0}^{m-d/2-1/2}f_{2j}r^{2j}}
+f_{2m-d}r^{2m-d} +\sum_{k>2m-d}f_kr^k 
$$
and
$$
K_{m-d/2}(hr)(hr)^{m-d/2}=\displaystyle{\sum_{j=0}^{m-d/2-1/2}f_{2j}h^{2j}r^{2j}}
+f_{2m-d}h^{2m-d}r^{2m-d} +\sum_{k>2m-d}f_kh^kr^k. 
$$
If we hit this twice with $\epsilon(h)$, i.e. forming
$$
\|\epsilon(h)\|^2_{H^m(\R^d)}=
\epsilon(h)^x\epsilon(h)^yK(h\|x-y\|_2),
$$
all even terms with exponents
$2j<2q=2p+2s>2m-d$ go away \cite{schaback:2005-2}, 
and we are left with the polyharmonic part and
higher-order terms. The odd ones are all polyharmonic, and the even ones
remain only from exponent $2q=2p+2s>2m-d$ on, i.e. they behave like 
$h^{2m-d+1}$ or higher-order terms. The polyharmonic terms 
$f_{2m-d+2k}h^{2m-d+2k}r^{2m-d+2k}$ representing $BL_{m+k,d}$ 
require polynomial exactness 
of order $m-d/2+1/2+k$ which is satisfied for $0\leq k<q-m+d/2$,
and double action of the error functional on these terms 
has a scaling law of $h^{2m+2k-2s-d}$. This means that the dominating term
is the one with $k=0$, and the squared error norm behaves 
like $h^{2m-d-2s}$ as in the $BL_{m,d}$ case.

Now we treat even dimensions, and use the expansion
$$
K_{m-d/2}(r)r^{m-d/2}=\displaystyle{\sum_{j=0}^{\infty}f_{2j}r^{2j}}
+g_{2m-d}\log(r)r^{2m-d}+ \log(r)\displaystyle{\sum_{2k>2m-d}g_{2k}r^{2k}}
$$
up to constant factors. With scaling, it reads as
$$
\begin{array}{rcl}
&&K_{m-d/2}(hr)h^{m-d/2}r^{m-d/2}\\
&=&
\displaystyle{\sum_{j=0}^{\infty}f_{2j}h^{2j}r^{2j}}
+g_{2m-d}\log(hr)h^{2m-d}r^{2m-d}
+ \log(hr)\displaystyle{\sum_{2k>2m-d}g_{2k}h^{2k}r^{2k}}\\
&=&
\displaystyle{\sum_{j=0}^{\infty}f_{2j}h^{2j}r^{2j}}
+g_{2m-d}h^{2m-d}\log(r)r^{2m-d}+g_{2m-d}\log(h)h^{2m-d}r^{2m-d}\\
&&
+ \displaystyle{\sum_{2k>2m-d}g_{2k}h^{2k}r^{2k}\log(r)} 
+\displaystyle{\sum_{2k>2m-d}g_{2k}h^{2k}\log(h)r^{2k}}\\
\end{array}
$$
We now have $q=p+s\geq 2m-d+2$ and 
hitting the scaled kernel twice will annihilate all even powers up to
and including exponents $2j<2q=2p+2s\geq 2m-d+2$, i.e. the remaining even powers
scale like $h^{2m-d+2}\log(h)$ or higher.  The rest is a sum of 
polyharmonic kernels $H_{m+k,d}$ for $k\geq 0$, and we know the scaling laws
of them, if the stencil has enough polynomial exactness.  Again, the term
with $k=0$ is the worst case, leading to a summand of type $h^{2m-d-2s}$
in the squared norm of the error that cannot be cancelled by the other terms of
higher order. 
\end{proof}
\section{Stencil Convergence}\label{SecSC}
Here, we prove that the renormalized weights of the optimal
non-scalable approximations in Sobolev space  converge to 
the weights of a scalable stencil. 
\begin{theorem}\label{theASA}
Consider the $W_2^m(\R^d)$-optimal approximation weights $a^*(h)$ on a 
set $X\subset\R^d$ for a functional of scaling order $s$. Assume that
$X$ allows a unique scalable stencil with weights $\hat a$ 
that is exact on polynomials
of order $q$. Then
$$
\|a^*(h)h^s-\hat a\|_\infty \leq Ch^{m-q+1-d/2}
$$
if $m-d/2 < q$, and
$$
\|a^*(h)h^s-\hat a\|_\infty \leq Ch^{1}
$$
if $m-d/2 \geq q$.
\end{theorem} 
\begin{proof} We consider the uniquely solvable 
system of polynomial exactness as
$$
\sum_{j=1}^M\hat a_jx_j^\alpha=\lambda(x^\alpha),\;0\leq |\alpha|<q
$$ 
and in scaled form as 
$$
\sum_{j=1}^Mh^{-s}\hat a_j(hx_j)^\alpha=\lambda(x^\alpha)
,\;0\leq |\alpha|<q
$$
which is the unscaled system where the equation for $x^\alpha$ is multiplied by
$h^{|\alpha|-s}$, namely
$$
\sum_{j=1}^Mh^{-s}\hat a_j(hx_j)^\alpha=h^{|\alpha|-s}\lambda(x^\alpha)=\lambda(x^\alpha)
,\;0\leq |\alpha|<q
$$
which is no contradiction because scaling order $s$ implies
$\lambda(x^\alpha)=0$ for $|\alpha|\neq s$.
Then we insert the rescaled optimal Sobolev weights into the unscaled system 
to get
\begin{equation}\label{eqlinsyscheck}
\begin{array}{rcl}
&&
h^s\sum_{j=1}^Ma_j^*(h)x_j^\alpha\\
&=&
h^{s-|\alpha|}\sum_{j=1}^Ma_j^*(h)(hx_j)^\alpha\\
&=&
h^{s-|\alpha|}\lambda_{a^*(h),hX}(x^\alpha)\\
&=&
h^{s-|\alpha|}(\lambda_{a^*(h),hX}(x^\alpha)-\lambda(x^\alpha))+h^{s-|\alpha|}
\lambda(x^\alpha)\\
&=&
h^{s-|\alpha|}(\lambda_{a^*(h),hX}(x^\alpha)-\lambda(x^\alpha))+
\lambda(x^\alpha)
\end{array} 
\end{equation}
and
$$
\begin{array}{rcl}
\sum_{j=1}^M(h^sa_j^*(h)-\hat a_j)x_j^\alpha
&=&h^{s-|\alpha|}(\lambda_{a^*(h),hX}(x^\alpha)-\lambda(x^\alpha)).
\end{array}
$$
If we insert the convergence rate $m-s-d/2$ for the optimal
Sobolev approximation in the case $m-s-d/2  <   q-s$ 
or $m-d/2  <  q$, the right-hand side of this system
converges to zero  with rate $m-|\alpha|-d/2\geq m-(q-1)-d/2\geq 1$
and this implies
\begin{equation}\label{eqstenrate}
h^sa_j^*(h)-\hat a_j ={\cal O}(h^{m-(q-1)-d/2}) \hbox{ for }h\to 0. 
\end{equation}
If we have $m-d/2\geq  q$, we insert the rate $q-s$ 
and get the rate $q-|\alpha|\geq 1$ for the right-hand side.
\end{proof} 
\section{Examples}\label{SecExa}
First, we demonstrate numerically that the convergence 
rate 
$$
\min(m-d/2-s,q_{max}(\lambda,X)-s)
$$
for approximations in $W_2^m(\R^d)$ 
to functionals $\lambda\in W_2^m(\R^d)^*$ with scaling order $s$ 
is optimal, even among unscaled
approximations. This was verified in many cases
including dimensions 2 and 3
using MAPLE$^\copyright$ with
extended precision.
The number of decimal digits had to be beyond 100 in extreme situations.
All the loglog plots of $\|\epsilon(h)\|_{W_2^m(\R^d)}$ versus
$h$ show the standard linear behaviour for $h\to 0$, 
if enough decimal digits are used
and if started with small $h$ values. Therefore,
they are suppressed here. Instead, we present convergence 
rate estimates by plotting 
$$
\dfrac{\log(\|\epsilon_{h_{i+1}}\|_{W_2^m(\R^d)})-\log(\|\epsilon_{h_i}\|_{W_2^m(\R^d)})}%
{\log(h_{i+1})-\log(h_i)}
$$
against $h_i$.

For a specific case, we take $M=18$ random points in 2D and approximate
the Laplacian. Then $s=2$ and $q_{max}(\lambda,X)=5$ leading to the
expected convergence rate $\min(m-3,3)$ as a function of smoothness. 
Figure \ref{figsobopton18pts} shows the cases $m=3.75$ and $m=6.25$
with the expected rates $0.75$ and 2, respectively. 
These correspond to situations where either smoothness $m$ or size of $X$
restrict the convergence rate. 
\begin{figure}[hbt] 
\begin{center}
\includegraphics[width=6cm,height=6cm]{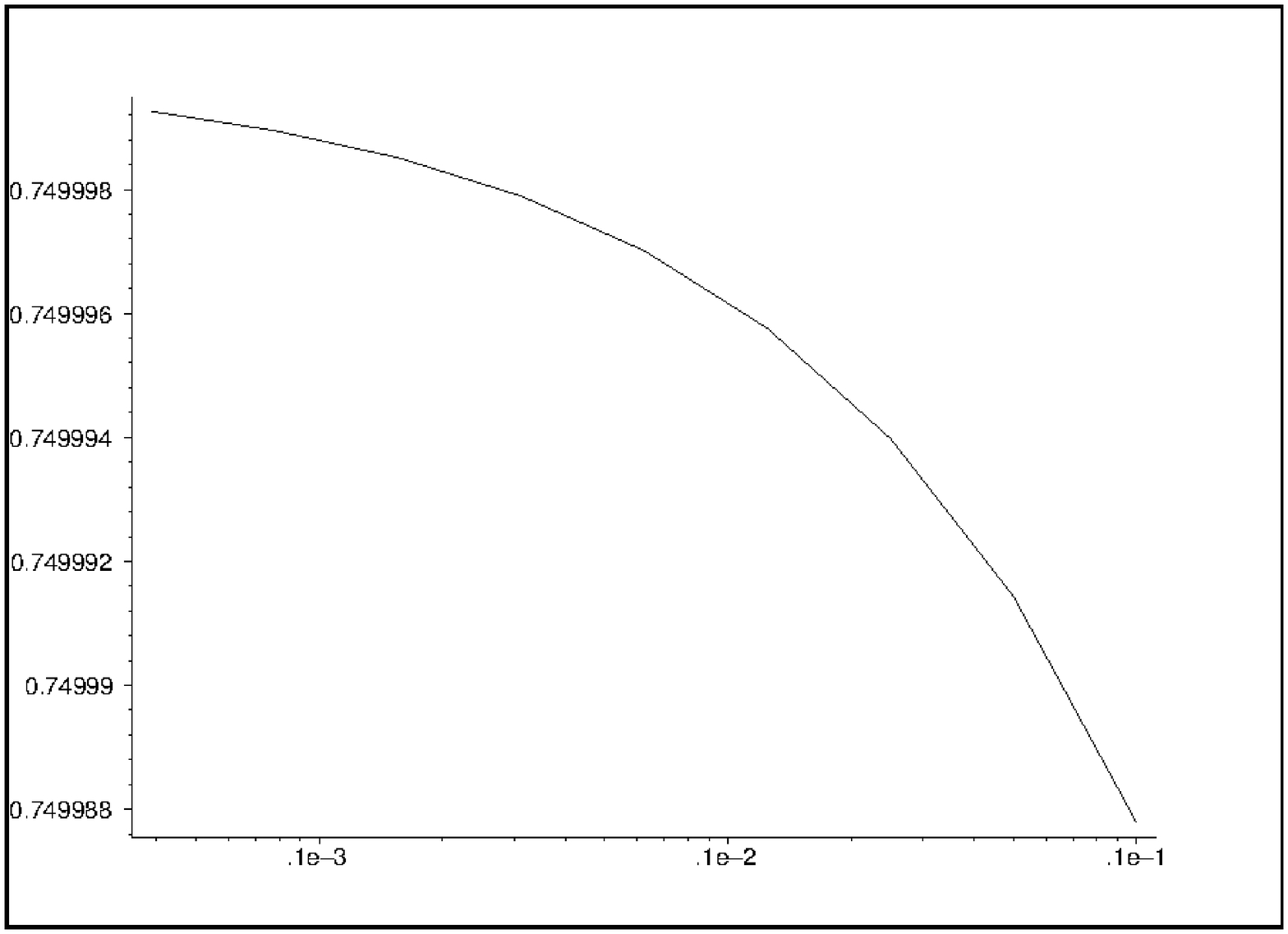}
\includegraphics[width=6cm,height=6cm]{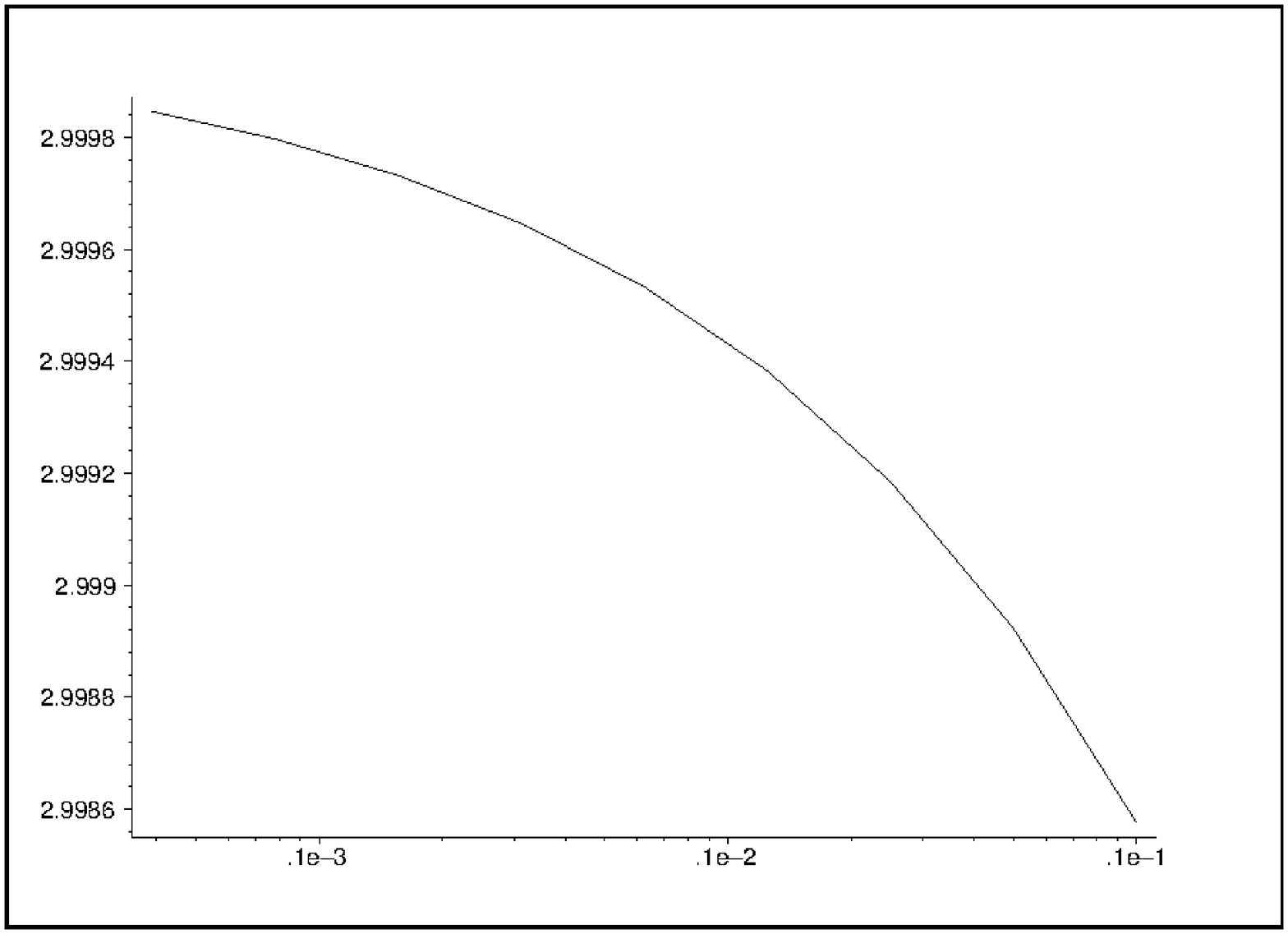}
\caption{Convergence rate estimates of the optimal $\epsilon_h$ in $W_2^{3.75}(\R^2)$ 
and $W_2^{6.25}(\R^2)$ 
approximating the 
Laplacian on $18$ general 
points as function of $h$\label{figsobopton18pts}}
\end{center}
\end{figure}

For illustration of the optimal compromise situation in 
\eref{eqcompro}, Figure 
\ref{figcompro1} shows the 
convergence
rate 1 for approximation of the Laplacian in 3D on only 10 points
in general position assuming smoothness $m=4.5$.
By Table \ref{tabSobRates} we expect a convergence rate between
  $m-s-d/2-\epsilon=1-\epsilon$ and 1 for all $\epsilon >0$ when using
  polynomial exactness order $q=m-d/2=3$, but the true optimal
  convergence could be like $h\log(h)$. The issue cannot be visually decided. 
\begin{figure}[hbt] 
\begin{center}
\includegraphics[width=7cm,height=7cm]{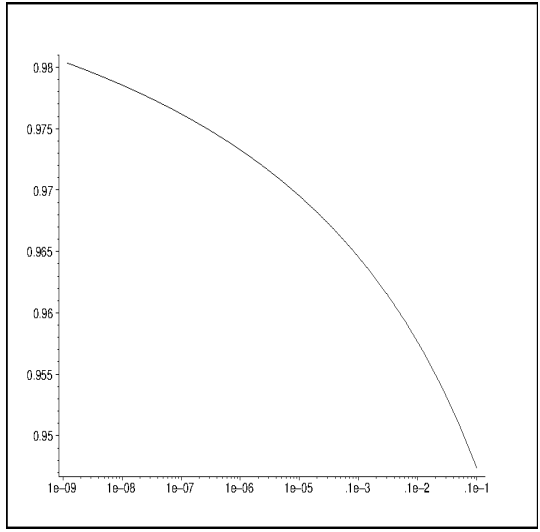}
\caption{Convergence rate estimates of the optimal $\epsilon_h$ in $W_2^{4.5}(\R^3)$ approximating
the Laplacian on $10$ general points as function of $h$%
\label{figcompro1}}
\end{center}
\end{figure}

Test runs with the scalable approximations based on polynomial exactness
show exactly the same behaviour, since they have the same convergence rate. 
To illustrate the ratio between the 
errors of scalable 
polyharmonic stencils and unscaled optimal approximations, 
Figure \ref{Figquot}
shows the error ratio in the 2D equilibrium case with 10 points and $m=q=4$,
tending to 1 for $h\to 0$. The same remark as for the $m=4.5,\,d=3$
case applies here.
\begin{figure}[hbt] 
\begin{center}
\includegraphics[width=7cm,height=7cm]{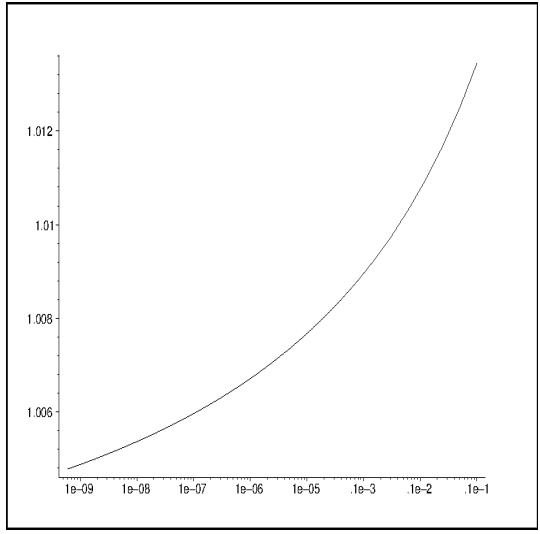}
\caption{Quotient between errors of polyharmonic and optimal
Sobolev approximations as functions of $h$%
\label{Figquot}}
\end{center}
\end{figure}

To deal with the special situation of $m-d/2$ being an integer 
in Corollary \ref{corSecCase2}
via polyharmonic kernels, we
take 6 points in $\R^2$ with $q=q_{max}=3$ for the Laplacian
with optimal convergence rate $m-2-d/2=1$ for $m=4$.
Working in $BL_{4,2}$ would need 10 points.  
A unique scalable stencil is obtained from $BL_{m',2}$ with
polynomial exactness order
$q(m',2)=3$ for all $3\leq m'<4$ and the convergence rate 
is at least $m-s-d/2-\epsilon=1-\epsilon$ for all $\epsilon >0$
by Table \ref{tabSobRates}.
The corresponding convergence rate estimate for $m'=3.5$  is in Figure
\ref{Figsoboptm4ph3p5on6pts}, and there is no visible
  $\log(h)$ factor.
\begin{figure}[hbt] 
\begin{center}
\includegraphics[width=7cm,height=7cm]{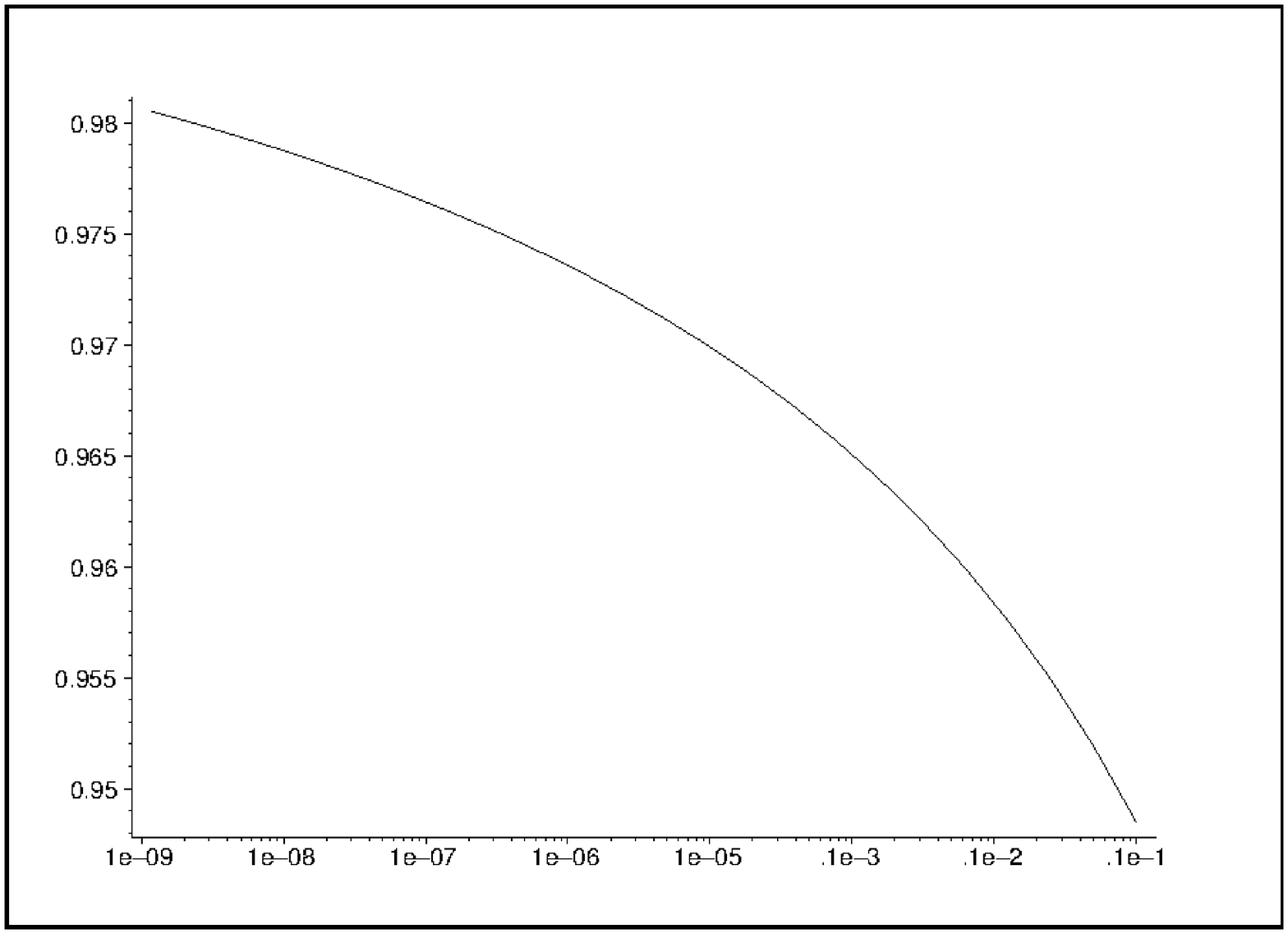}
\caption{Convergence rate 
estimate for the error norm of $\epsilon_h$ in $W_2^{4}(\R^2)$ approximating
the Laplacian on $6$ general points 
by a stencil of polynomial exactness of order 3
\label{Figsoboptm4ph3p5on6pts}}
\end{center}
\end{figure}

To see whether a $\log(h)$ term can be present in the situation
of integer $q=m-d/2$,  we take $m=d=2,\;q=1,\;s=0,$ 
i.e. interpolation. We need just a single point
$x\in\R^2$ with $\|x\|_2=1$ for exactness on constants. The kernel is
$\phi(r)=rK_1(r)=1+\frac{1}{2}r^2\log r +{\cal O}(r^2)$ with $\phi(0)=1$.
The optimal recovery for $\lambda(u)=u(0)$ from $u(hx)$ is the kernel
interpolant, i.e. $u(hx)\phi(\|\cdot-hx\|_2)$, and the approximation error
is
$$
u(0)-u(hx)\phi(\|hx\|_2)=u(0)-u(hx)\phi(h).
$$
In the dual of $W_2^2(\R^2)$ the square
of the norm of the error functional is
$$
\begin{array}{rcl}
  \|\delta_0-\phi(h)\delta_{hx}\|^2_{{W_2^2}^*(\R^2)}
  &=& \phi(0)-\phi(h)^2\\
  &=& -h^2\log(h)+{\cal O}(h^2)\\
\end{array}
$$
due to MAPLE. Since the standard error bound
$$
|u(0)-u(hx)\phi(h)|\leq
\|\delta_0-\phi(h)\delta_{hx}\|_{{W_2^2}^*(\R^2)}
\|u\|_{W_2^2(\R^2)}
$$
is sharp, and since we constructed the optimal recovery, we have that
the convergence for $q=1$ is  only $h|\log(h)|^{1/2}$ and not
like the optimal behaviour
$h^{m-0-d/2}=h$ in Sobolev space $W_2^2(\R^2)$.
To reach the optimal rate,
we need a polynomial exactness order $q\geq 2$
by  Table \ref{tabSobRates}, i.e. at least three
non-collinear points. For curiosity, note that
the above analysis works for all even dimensions,
provided that smoothness $m=1+d/2$ is varying accordingly.

The suboptimal nearest-neighbor interpolation by constants has
$$
\begin{array}{rcl}
  \|\delta_0-\delta_{hx}\|^2_{{W_2^2}^*(\R^2)}
  &=& 2-2\phi(h)\\
  &=& -h^2\log(h)+{\cal O}(h^2)
\end{array}
$$
and a more exact expansion via MAPLE
shows that this is larger than the squared
error for optimal one-point interpolation
in $W_2^{1+d/2}(\R^d)$ by ${\cal O}(\log^2(h)h^4)$.

In several numerical examples we verified the stencil convergence
proven in Theorem \ref{theASA}, but the observed convergence rates 
turned out to be better
than the proven ones. In particular, choosing 15 points in general position
in $\R^2$ with $q=5$ led to a convergence rate
$\min(2,2m-10)$ for $m\geq 5$ 
instead of $\min(1,m-5)$ in Theorem \ref{theASA}.
This seems to be a consequence of {\em superconvergence}
\cite{schaback:1999-1,schaback:2016-5}, but needs further 
work.

We now check approximation
of the Laplacian in the native space of the Gaussian
in Figure \ref{figgauss30}. This should behave like $m=\infty$
in \eref{eqOptRate} and thus show a convergence rate $q_{max}(\lambda,X)-s$.
We used 256 decimal digits for that example 
and took a set of 30 random points in 2D. Then $q_{max}(\Delta,X)=7$
and the observed convergence rate is indeed $q_{max}-s=5$. Furthermore,
this rate is attained already for
a scalable stencil that is polynomially exact of order $7$ on these points.
We chose the optimal scalable polyharmonic stencil in $BL_{7,2}$ for this, and 
the ratio of the error norms was about 5. See \cite{larsson-et-al:2013-1}
for a sophisticated way to circumvent the instability of calculating optimal 
non-scalable stencils for Gaussian kernels, but this paper suggests
to use scalable stencils calculated via polyharmonic kernels instead.
\begin{figure}[hbt] 
\begin{center}
\includegraphics[width=6cm,height=6cm]{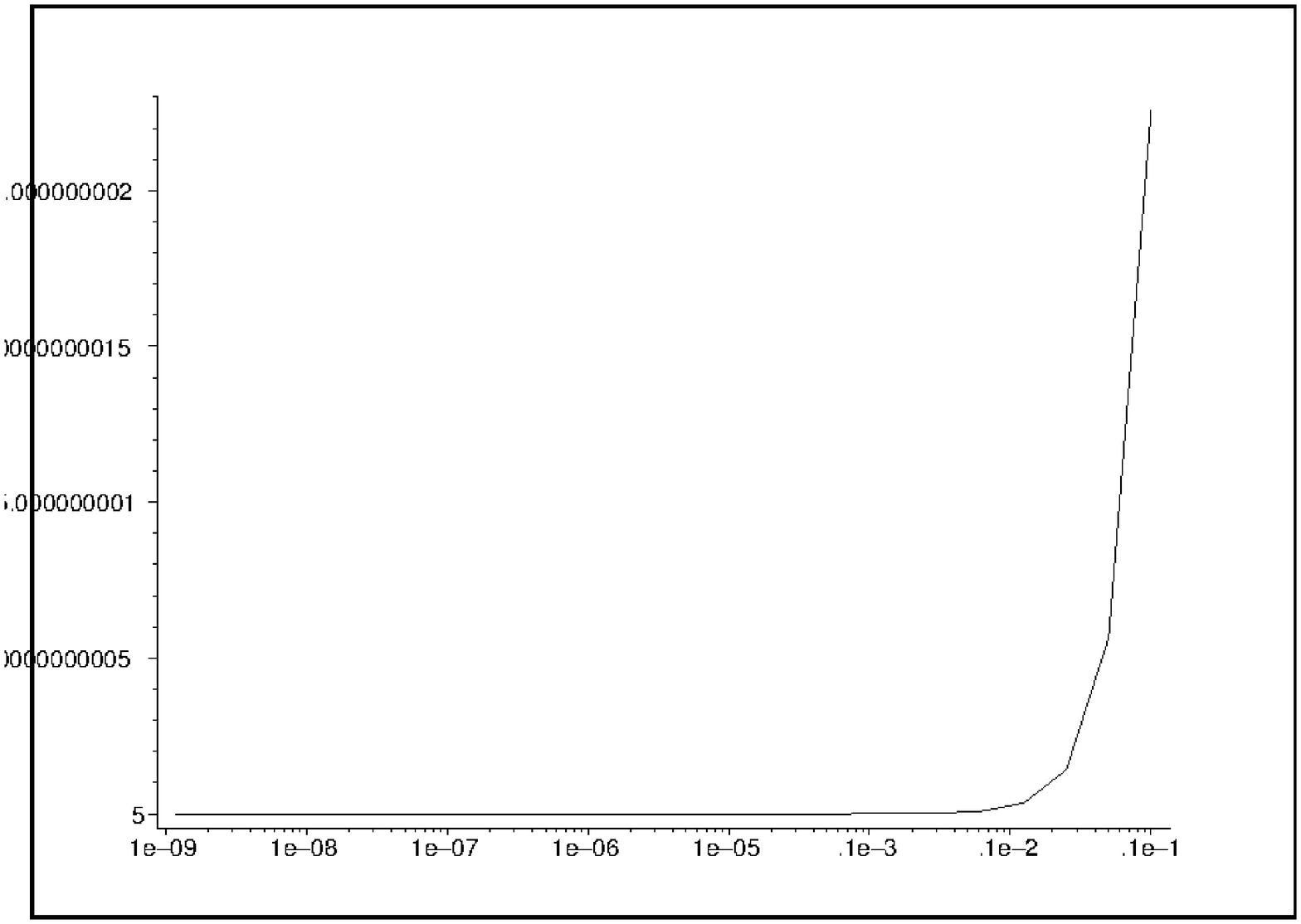}
\includegraphics[width=6cm,height=6cm]{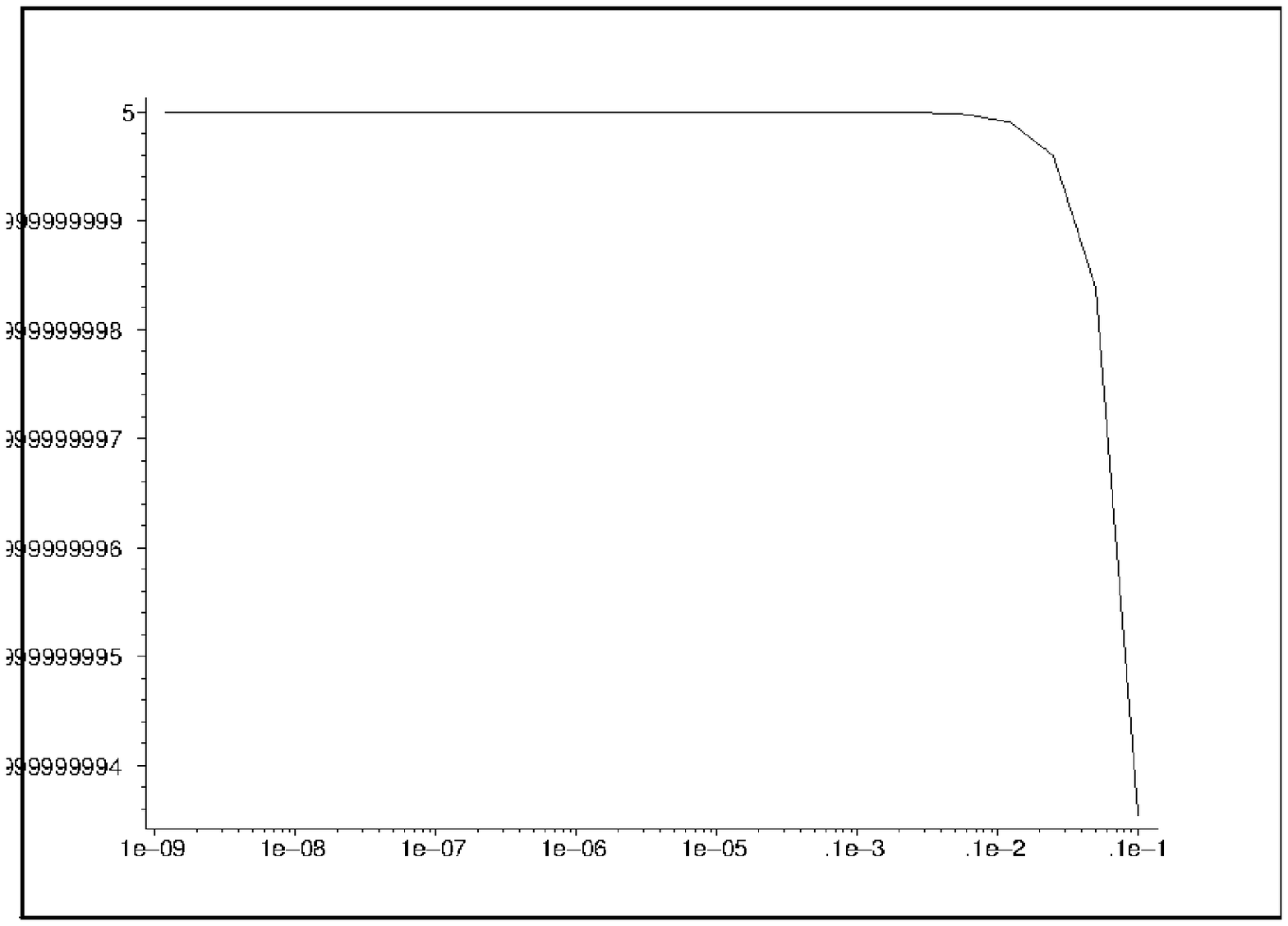}
\caption{Gaussian native space convergence rate estimates 
for the error norms of the optimal and
a polynomially exact stencil of order $7$,  
approximating the 
Laplacian on $30$ general 
points, as function of $h$\label{figgauss30}}
\end{center}
\end{figure}

We finally compare with approximations 
that optimize weights under the constraint of a fixed 
polynomial exactness \cite{davydov-schaback:2016-1}.

The three point sets $X1,\;X2$, and $X3$ of \cite{davydov-schaback:2016-1}
have 32 points in $[-1,+1]^2$ each, and the maximal possible 
order of polynomial reproduction in 2D is 7, if the geometry of the
point set allows it. 
If everything works fine, this would result in convergence
of optimal order $5$ for the approximation of the Laplacian in Sobolev spaces
of order $m\geq 8$, while the optimal rate for smaller $m$ is $m-3$.

A simple 
Singular Value Decomposition of the 28x32 value matrix of polynomials of order 7
on these points reveals that the small singular values in the three cases are
like in Table \ref{svdtable}. This means that only $X1$ allows working for
exactness order 7 without problems, 
while $X2$ suggests order $6$ and $X3$ should still
work with order $5$. If users require higher polynomial
exactness orders (PEO), there is a risk of numerical instabilities.

To demonstrate this effect, Figure \ref{FigX2m8O7P7} shows 
what happens if both the polyharmonic and the minimal-weight approximations 
are kept at order 7 for the set $X2$. As Figure \ref{FigX2m8O6P7runs20}
will show, the optimal Sobolev approximation stays at rate 4 for larger $h$
and needs rather small $h$ to show its optimal rate 5. In
Figure \ref{FigX2m8O7P7}, both the polyharmonic and the minimal-weight
approximations perform considerably worse than the optimum. If we go to 
polynomial exactness order 6, we get Figure \ref{FigX2m8O6P6}, and now
both approximations are close to what the Sobolev approximation does, though the
latter is not at its optimal rate yet. In Figure \ref{FigX2m8O6P7runs20},
the polyharmonic approximation is forced to stay at exactness order 7, while
the weight-minimal
approximation is taken at order 6 to allow more leeway for weight optimization.
Now, in the same range as before, the weight-optimal approximation
clearly outperforms the polyharmonic approximation. The same situation
occurs on the set $X3$ under these circumstances, see   
Figure
\ref{FigX3m8O6P7runs20}. Thus, for problematic point sets,
the polyharmonic approximation should get as much  leeway as the
minimal-weight approximation.

The most sensible choice on $X3$ is to
fix the exactness orders to 5, and the results are in Figure 
\ref{FigX3m8O5P5}. Both approximations cannot compete with the 
convergence rate 4 that the Sobolev approximation shows in this range of $h$.
The latter is calculated using 128 digits and can still use the point set as one
that allows polynomial reproduction of order 6. The other two
approximations are calculated at 32 decimal digits and see the set
$X3$ as one that allows reproduction of order 5 only. 
To get back to a stable situation, we should lower the Sobolev smoothness 
to $m=6$ to get Figure \ref{FigX3m6O5P5}.  We then are back to a convergence
rate like $h^3$ in all cases.

\begin{table}[hbt]
\begin{center}
\begin{tabular}{|r||c|c|c|}
\hline
 Set & $>0.002$ & $\in [2.0e-8,3.6e-7]$ & $<5.0e-14$\\
\hline
X1 & 28 & 0 & 0 \\ 
X2 & 25 & 3 & 0 \\ 
X3 & 18 & 9 & 1 \\
\hline 
\end{tabular} 
\end{center}
\caption{Singular values for three point sets,
for polynomial reproduction of order 7 \label{svdtable}}  
\end{table}

\begin{figure}[hbt] 
\begin{center}
\includegraphics[width=6cm,height=6cm]{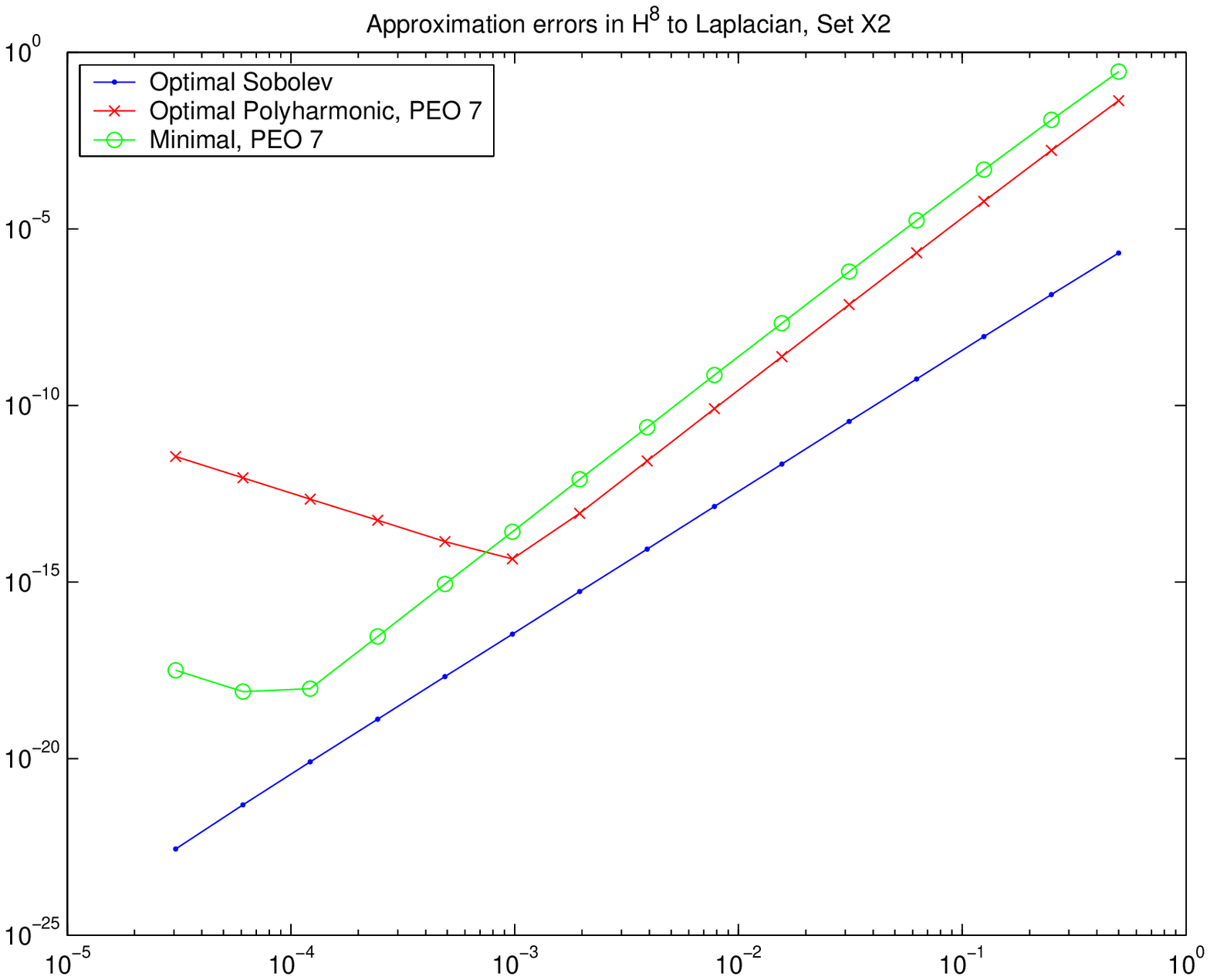}
\includegraphics[width=6cm,height=6cm]{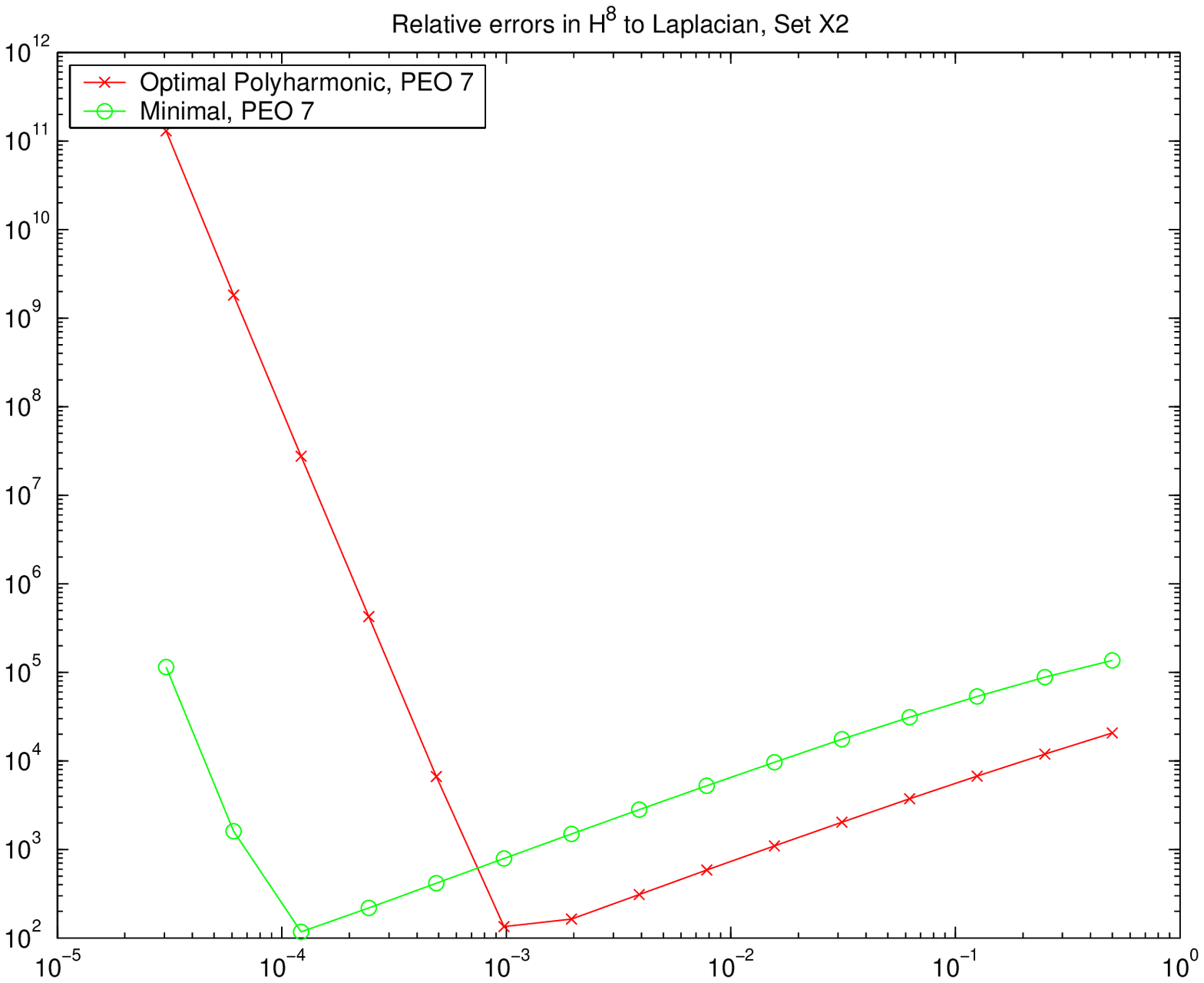}
\caption{Absolute and Sobolev-relative error norms in $W_2^{8}(\R^2)$
for approximations with polynomial exactness order (PEO) 7 on set $X2$
\label{FigX2m8O7P7}}
\end{center}
\end{figure}

\begin{figure}[hbt] 
\begin{center}
\includegraphics[width=6cm,height=6cm]{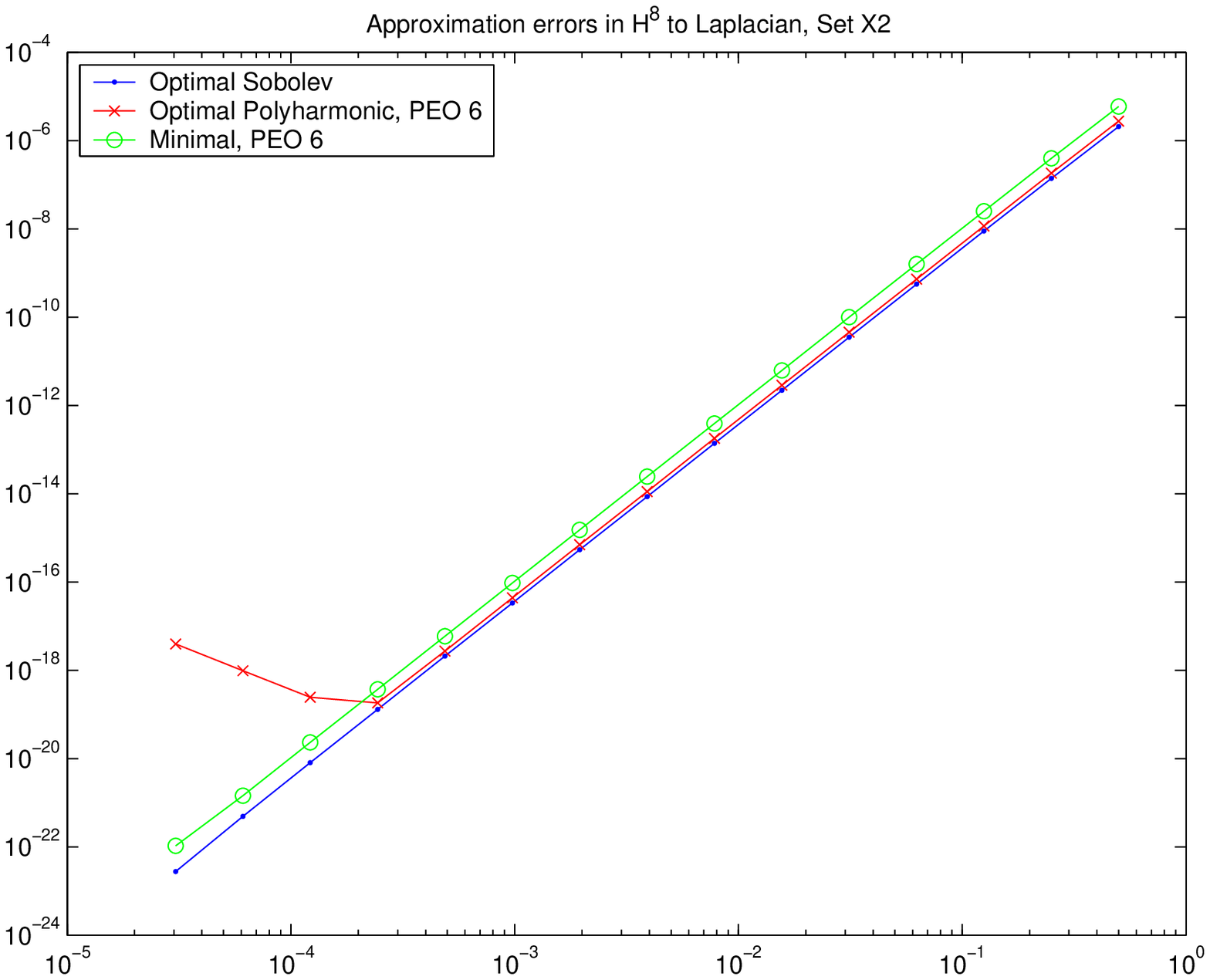}
\includegraphics[width=6cm,height=6cm]{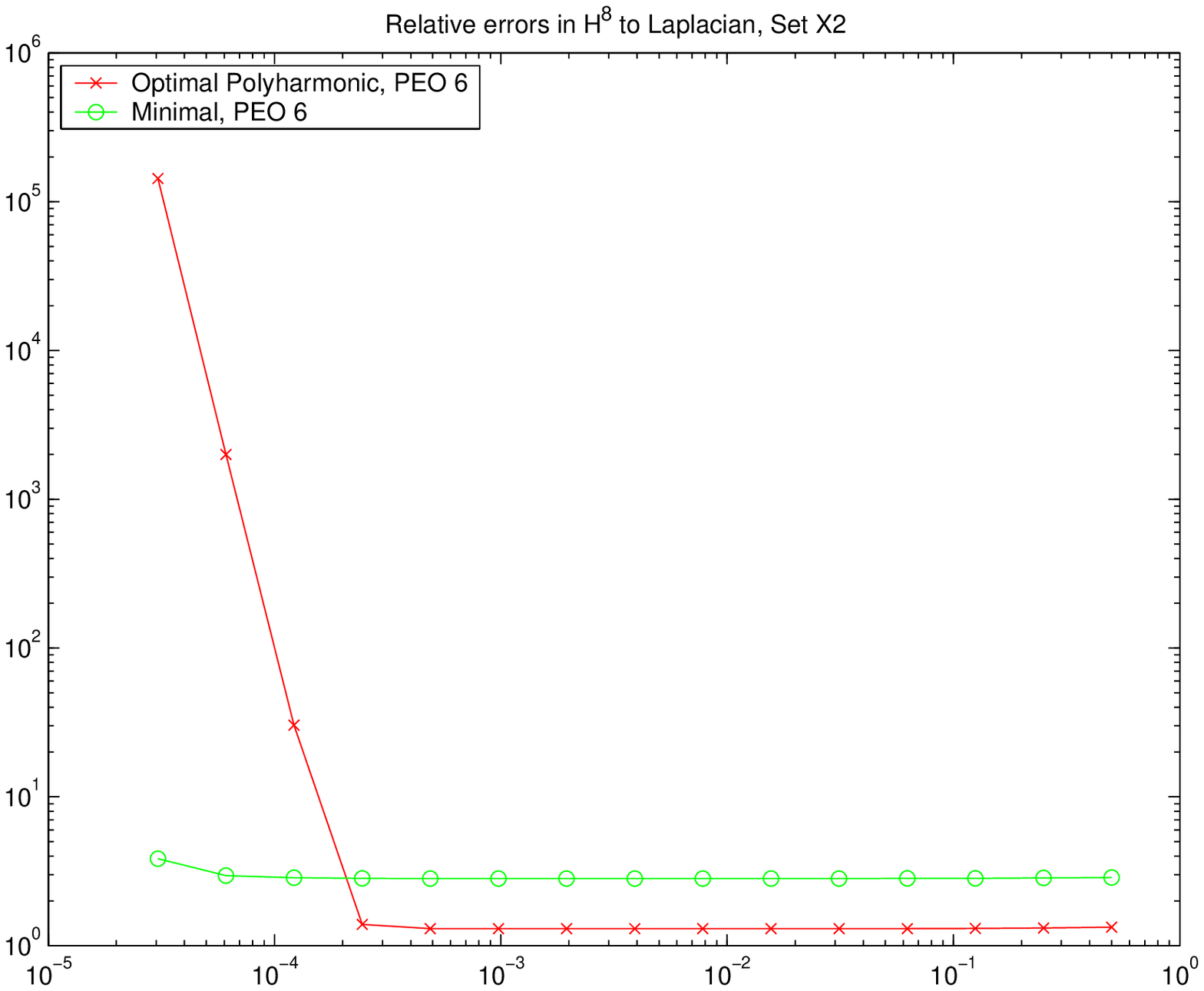}
\caption{Absolute and Sobolev-relative error norms in $W_2^{8}(\R^2)$
for approximations with polynomial exactness order 6 on set $X2$
\label{FigX2m8O6P6}}
\end{center}
\end{figure}

\begin{figure}[hbt] 
\begin{center}
\includegraphics[width=6cm,height=6cm]{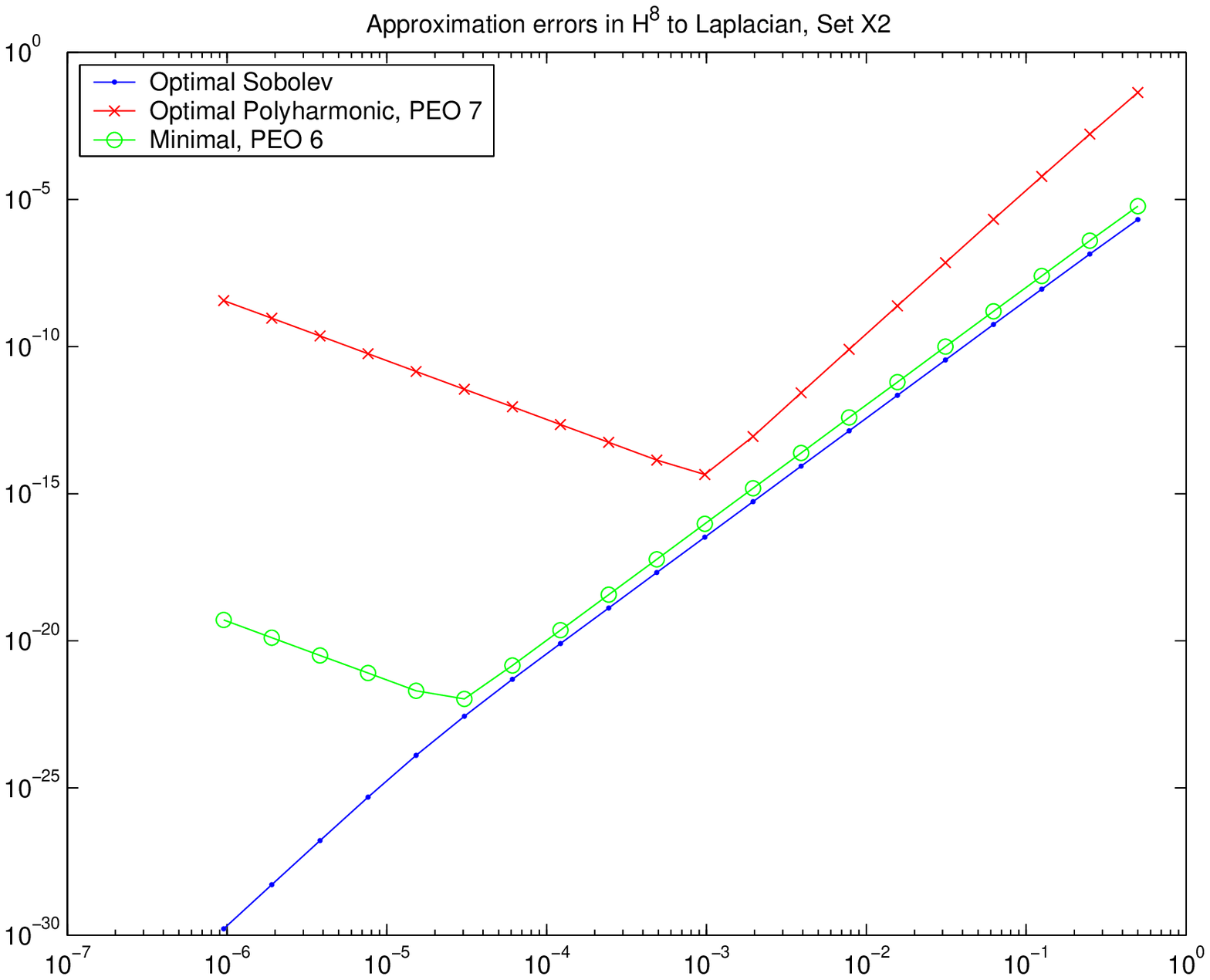}
\includegraphics[width=6cm,height=6cm]{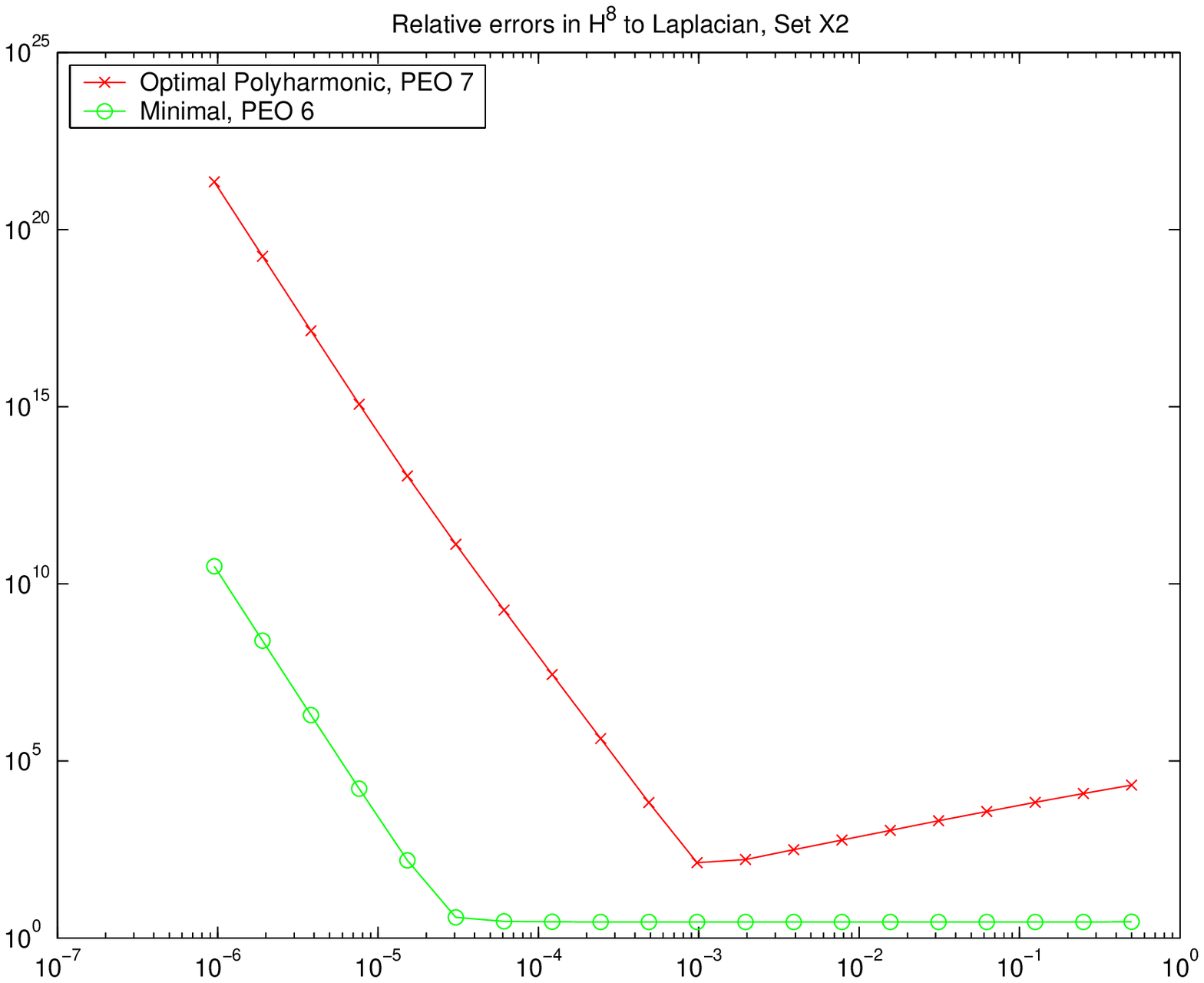}
\caption{Absolute and Sobolev-relative error norms in $W_2^{8}(\R^2)$
for polyharmonic approximation of order 7 
and minimal approximation of order 6 on set $X2$
\label{FigX2m8O6P7runs20}}
\end{center}
\end{figure}

\begin{figure}[hbt] 
\begin{center}
\includegraphics[width=6cm,height=6cm]{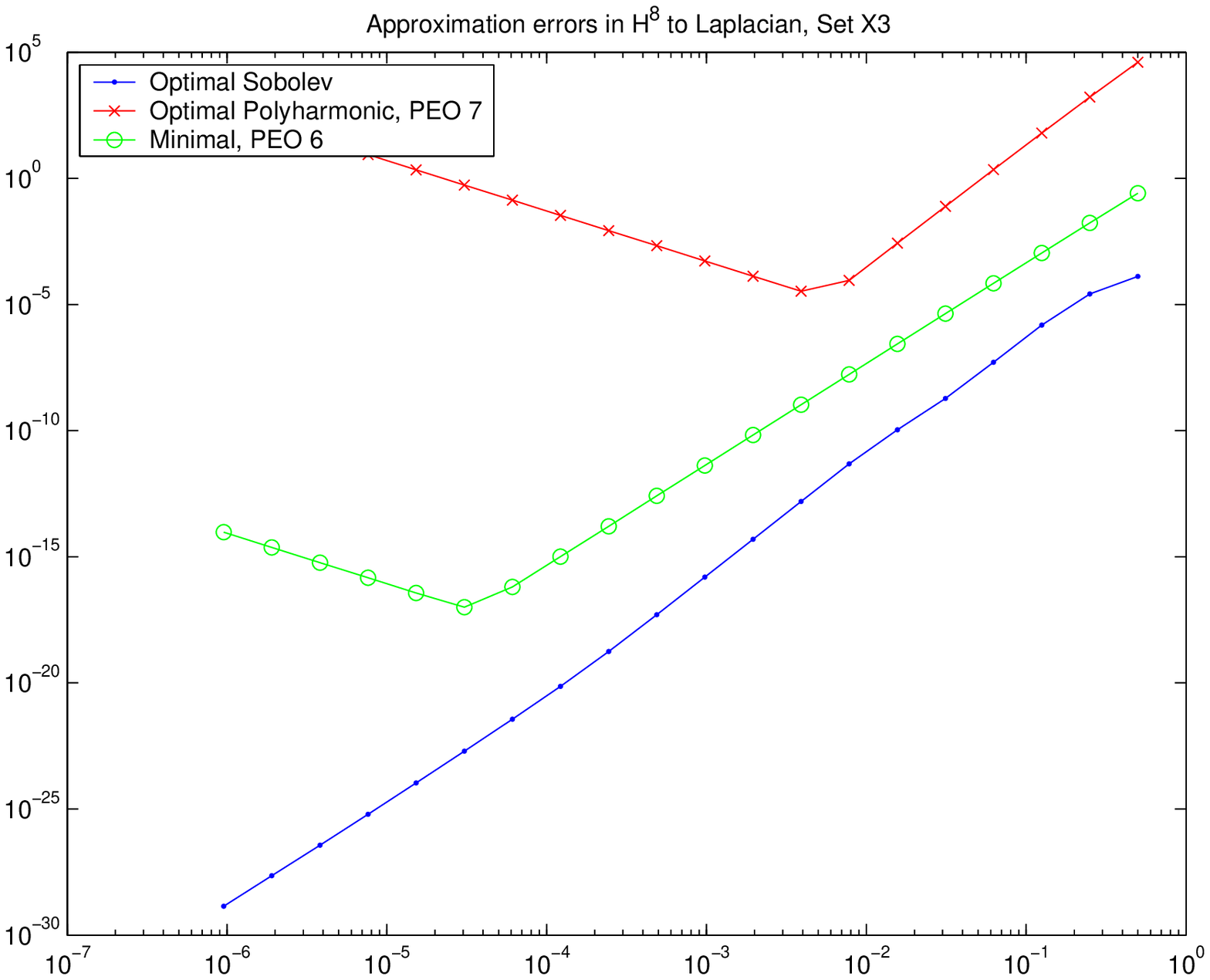}
\includegraphics[width=6cm,height=6cm]{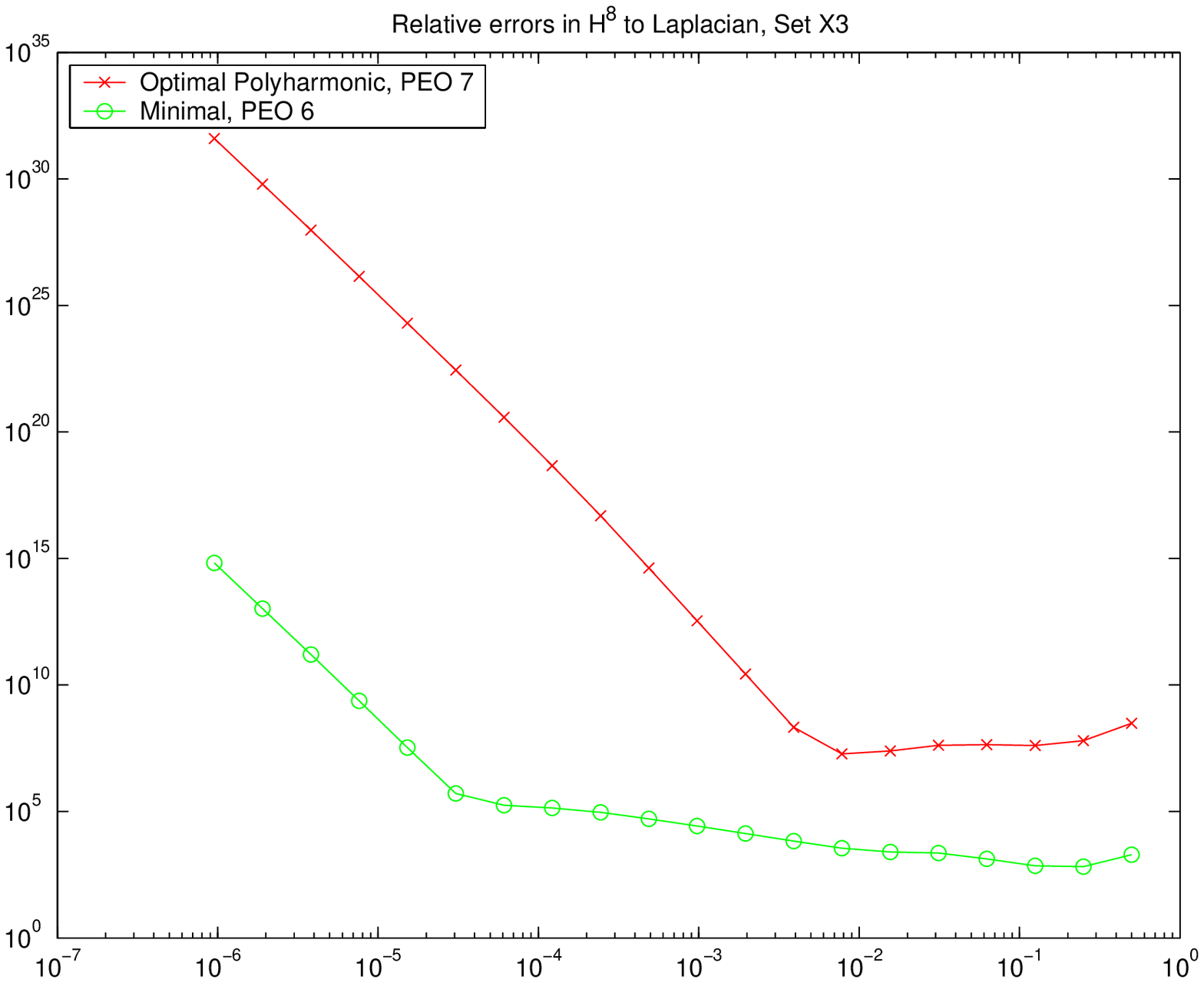}
\caption{Absolute and Sobolev-relative error norms in $W_2^{8}(\R^2)$
for polyharmonic approximation of order 7 
and minimal approximation of order 6 on set $X3$
\label{FigX3m8O6P7runs20}}
\end{center}
\end{figure}

\begin{figure}[hbt] 
\begin{center}
\includegraphics[width=6cm,height=6cm]{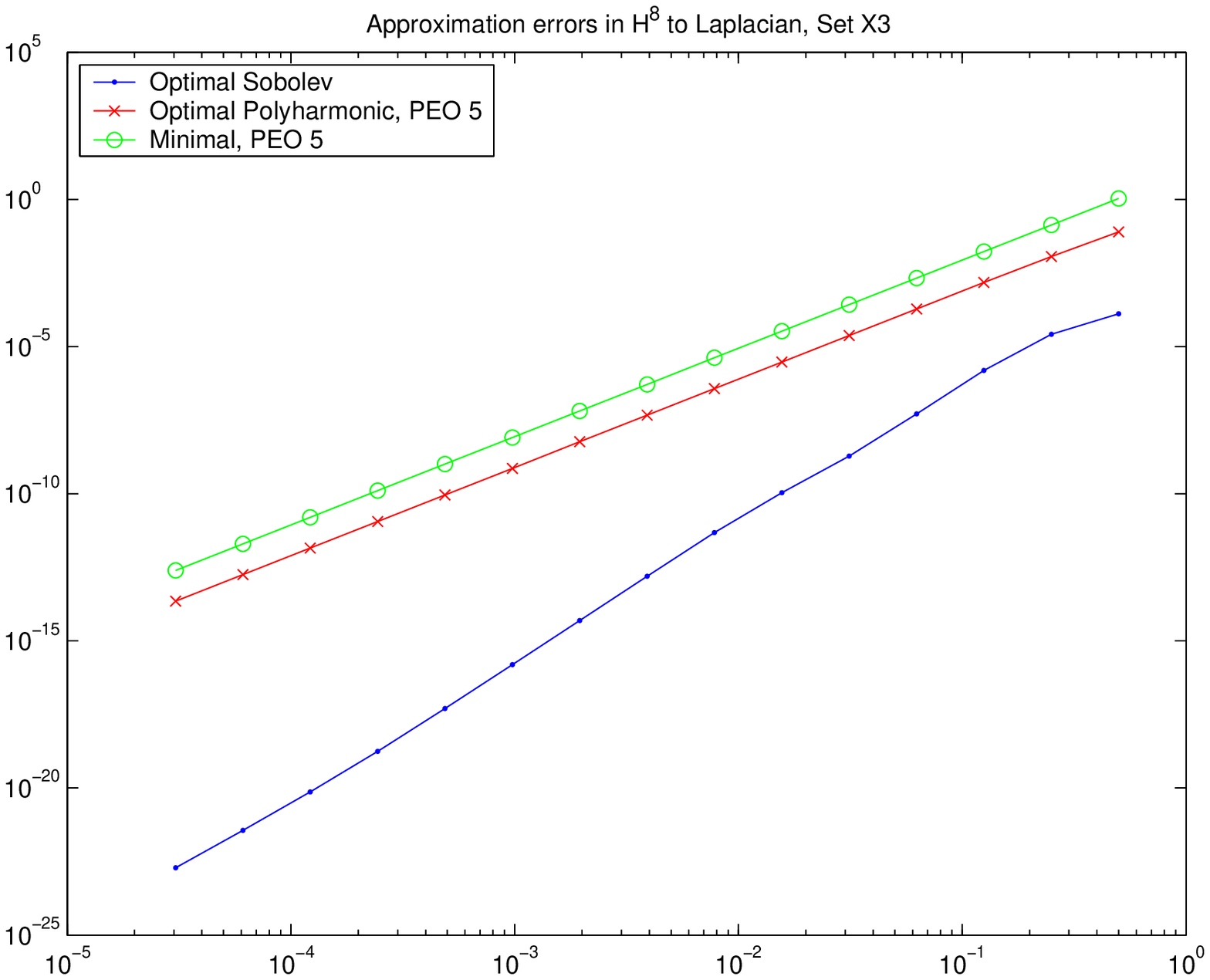}
\includegraphics[width=6cm,height=6cm]{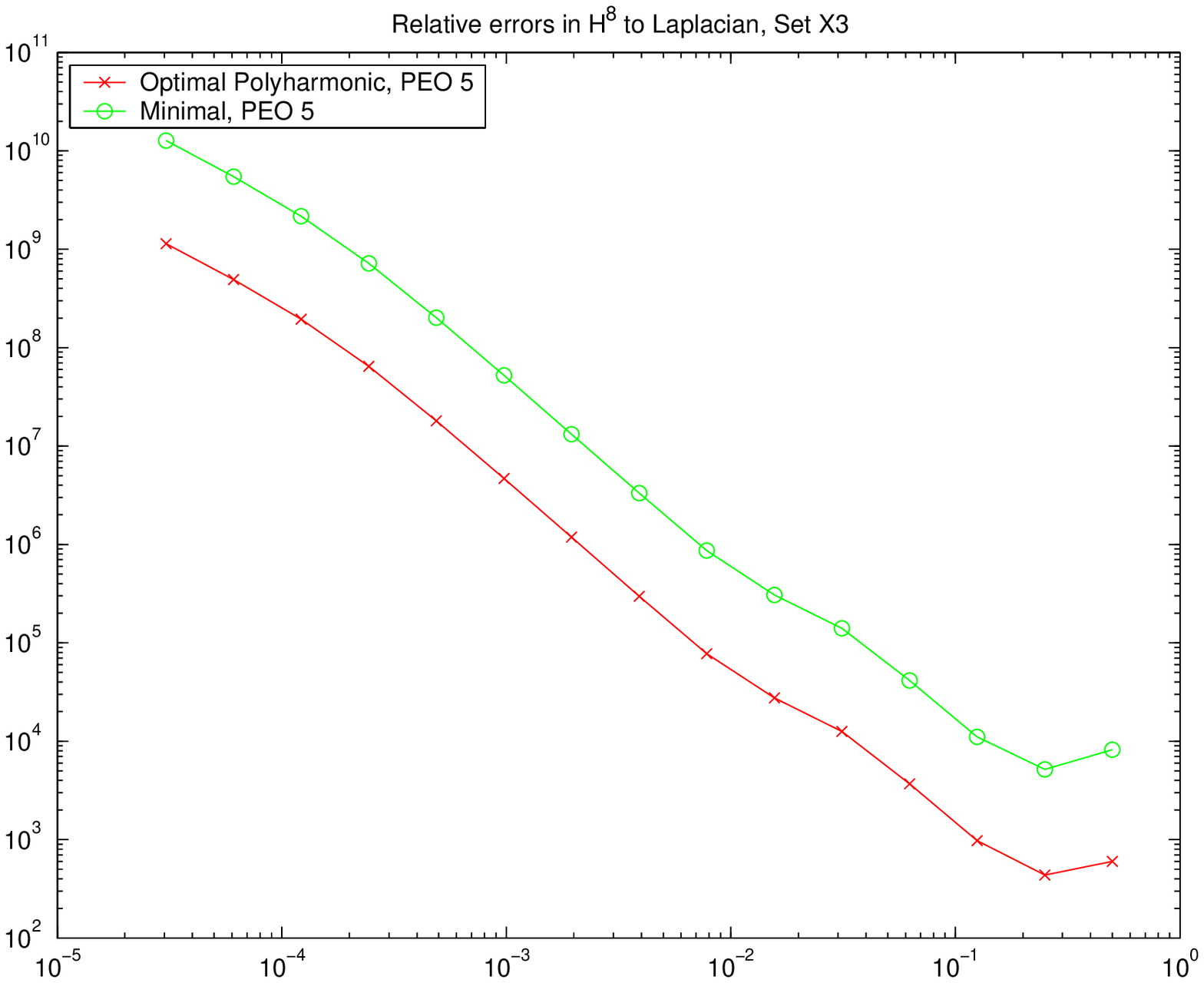}
\caption{Absolute and Sobolev-relative error norms in $W_2^{8}(\R^2)$
for polyharmonic 
and minimal approximation of order 5 on set $X3$
\label{FigX3m8O5P5}}
\end{center}
\end{figure}

\begin{figure}[hbt] 
\begin{center}
\includegraphics[width=6cm,height=6cm]{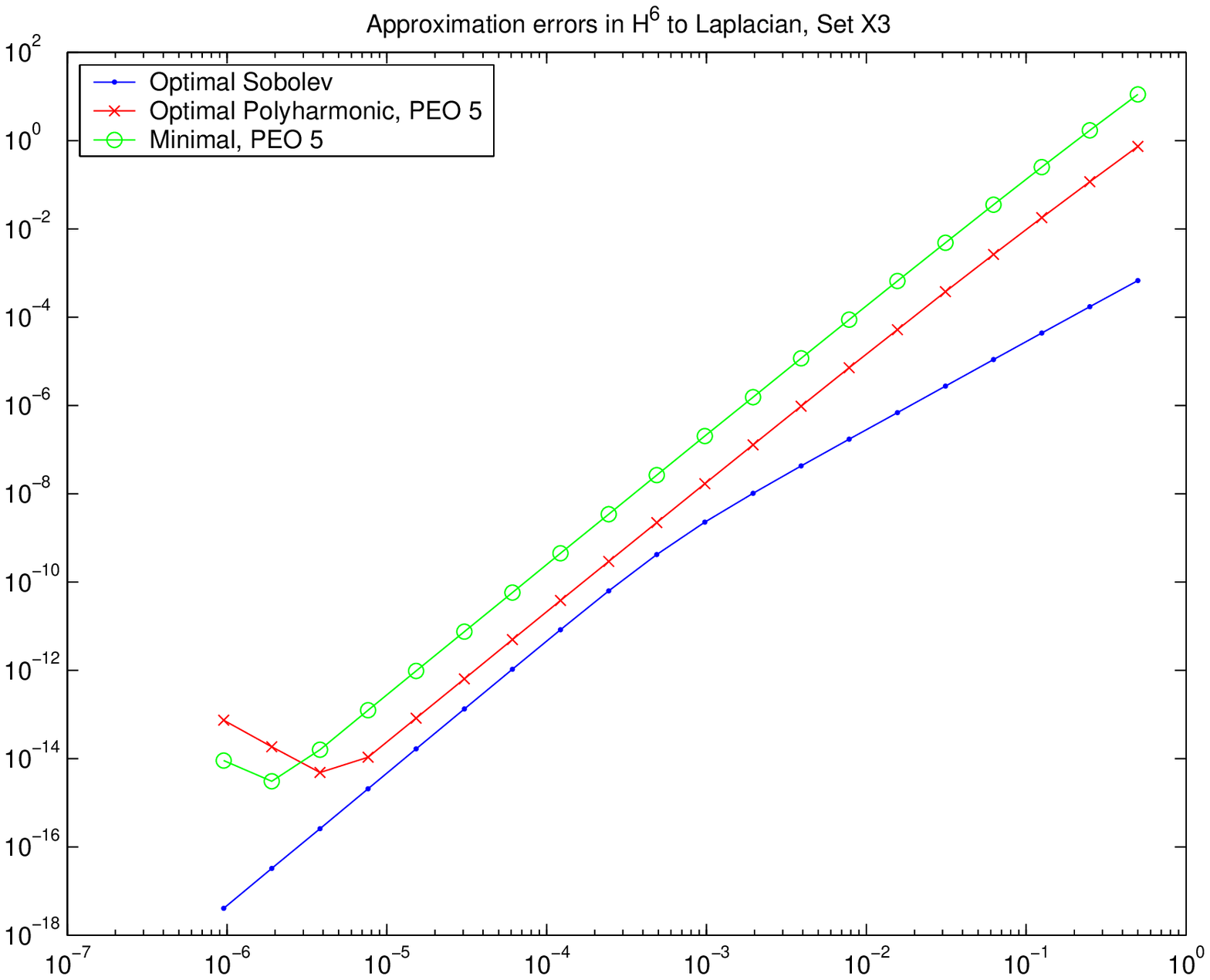}
\includegraphics[width=6cm,height=6cm]{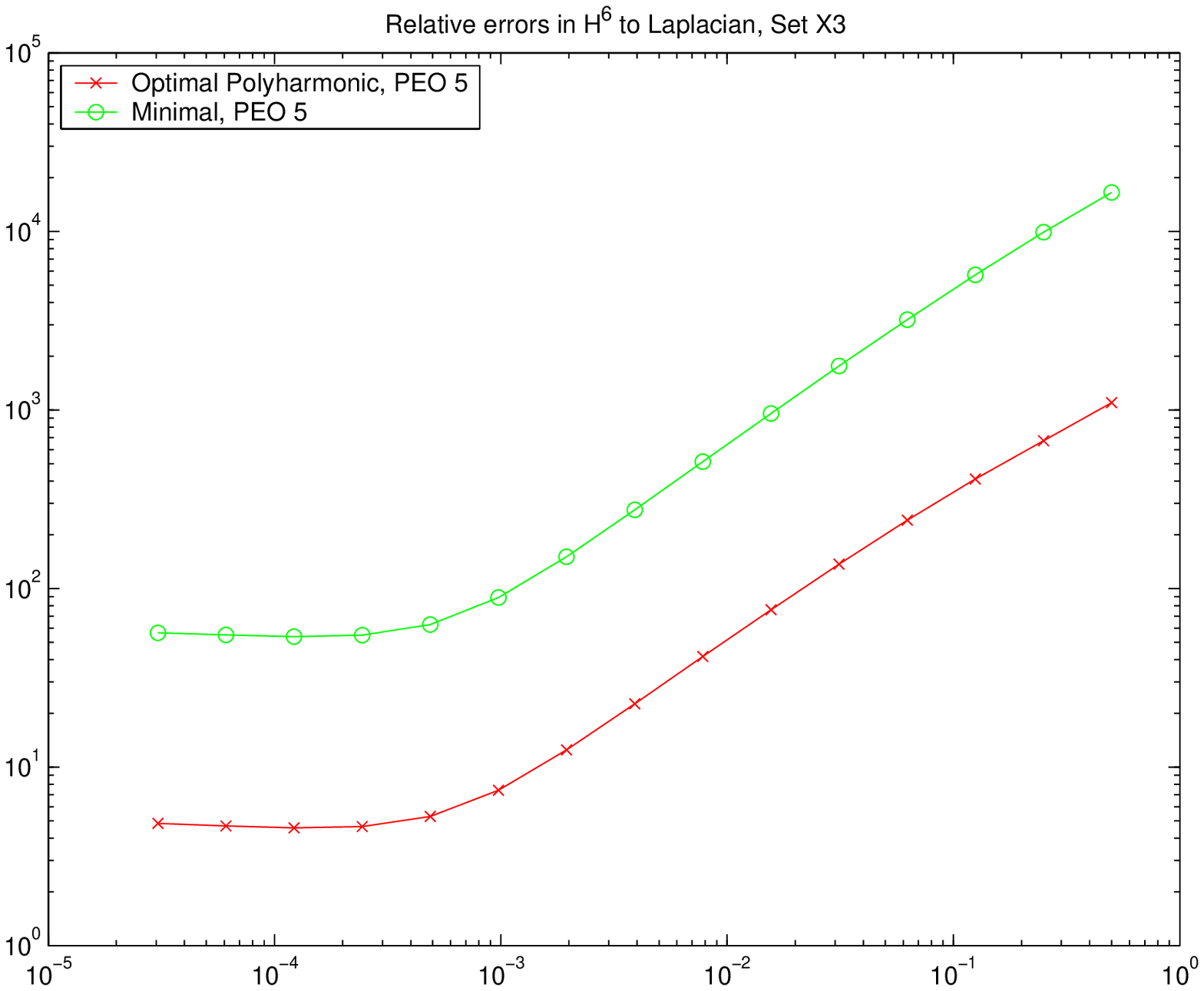}
\caption{Absolute and Sobolev-relative error norms in $W_2^{6}(\R^2)$
for polyharmonic 
and minimal approximation of order 5 on set $X3$
\label{FigX3m6O5P5}}
\end{center}
\end{figure}

\section{Summary and Outlook}\label{SecU}
We established the optimal convergence rate 
\eref{eqOptRate} of nodal approximations in 
Sobolev spaces and proved that it can be 
attained for {\em scalable} approximations with 
sufficient polynomial exactness. 
But we did not investigate the factors in front of the rates.
For highly irregular nodes, it might be reasonable to  
go for a smaller convergence rate, if the factor is much smaller 
than the one for the highest possible rate for that node configuration.
This requires an analysis of how to use the additional degrees of freedom,
and various possibilities for this are in 
\cite{davydov-schaback:2016-1}.  On point sets 
that are badly distributed, it pays off to avoid the highest possible
order of polynomial exactness, and to use the additional degrees of freedom
for minimization of weights along the lines of \cite{davydov-schaback:2016-1}
or to use optimal approximations by polyharmonic kernels at a  smaller 
order of polynomial exactness.

The kernels reproducing Sobolev spaces $W_2^m(\R^d)$ 
have expansions
into power series in $r=\|x-y\|_2$ that start with even powers of $r$ until 
the polyharmonic kernel $H_{m,d}$ occurs. This
shows that error evaluation in Sobolev spaces can be replaced
asymptotically by evaluation in Beppo-Levi spaces, and it 
suggests that the errors of optimal kernel-based
approximations should be close to the errors of optimal scalable 
stencils based on polyharmonic kernels. This occurred in
various experiments (see Figure \ref{Figquot}), 
but a more 
thorough investigation is needed.

Finally, the exceptional case $m-d/2\in \N$
  of the second row of Table \ref{tabSobRates} needs more attention.
  Approximating a functional with scaling order $s$
  by scalable stencils with the minimal polynomial exactness order
  $q=m-d/2$ leads to an unknown convergence behavior between rates 
  $m-s-d/2-\epsilon$ and the optimal rate $m-s-d/2$ that
  is guaranteed for order $q+1=m-d/2+1$. The
  convergence could be like ${\cal O}(h^{m-s-d/2}|\log(h)|^p)$, for instance,
  and we presented an example with $p=1/2$ for $m=d=2,\;s=0$.
\bibliographystyle{plain}

\end{document}